\documentclass[a4paper, english, 10pt]{smfart} 

\usepackage{placeins} 

\usepackage{setspace}
\usepackage[british]{babel}

\usepackage[OT2,T1]{fontenc}

\usepackage{amssymb} 
\usepackage{mathrsfs} 

\usepackage{stmaryrd}
\usepackage{amsthm}

\usepackage{amsmath}
\usepackage[latin1]{inputenc}

\allowdisplaybreaks[1]

\usepackage{graphicx}
\usepackage{manfnt} 

 \usepackage{booktabs} 

\usepackage[margin=3cm]{geometry}

\usepackage[all]{xy}

\usepackage{math tools} 

\usepackage{enumitem}

\usepackage[pagebackref=true, colorlinks=true, linkcolor=black, citecolor=black, urlcolor=black]{hyperref}
\usepackage{smfthm}

\usepackage[msc-links, alphabetic]{amsrefs}

\DeclareSymbolFont{cyrletters}{OT2}{wncyr}{m}{n}
\DeclareMathSymbol{\Sha}{\mathalpha}{cyrletters}{"58}

\newtheorem{theoA}{Theorem}

\newtheorem*{coro*}{Corollary}
\newtheorem*{conj*}{Conjecture}
\newtheorem*{lemm*}{Lemma}

\newcommand{\smallmat}[4]{\bigl(\begin{smallmatrix}#1&#2\\#3&#4\end{smallmatrix}\bigr)}
\providecommand{\twomat}[4]{\left(\begin{array}{cc}#1&#2\\#3&#4\end{array}\right)}
\providecommand{\smalltwomat}[4]{\left(\begin{smallmatrix}#1&#2\\#3&#4\end{smallmatrix}\right)}

\theoremstyle{definition}

\theoremstyle{remark}
\newtheorem{remark*}{Remark}

\NumberTheoremsIn{subsection}

\numberwithin{equation}{subsection}
\numberwithin{table}{subsection}

\newcommand{\Div}{\mathrm{Div}}

\newcommand{\ot}{\otimes}

\newcommand{\ts}{\times}

\newcommand{\beq}{\begin{equation}\begin{aligned}}
\newcommand{\eeq}{\end{aligned}\end{equation}}

\newcommand{\beqq}{\begin{equation*}\begin{aligned}}
\newcommand{\eeqq}{\end{aligned}\end{equation*}}

\newcommand{\lb}[1]{\label{#1}}

\newcommand{\one}{\mathbf{1}}
\newcommand{\Q}{\mathbf{Q}}
\newcommand{\qqq}{\mathbf{q}}

\newcommand{\GL}{\mathrm{GL}}

\renewcommand{\H}{\mathrm{H}}

\newcommand{\X}{\mathscr{X}}

\newcommand{\cE}{\mathscr{E}}

\newcommand{\cV}{\mathscr{V}}

\newcommand{\cP}{\mathscr{P}}

\newcommand{\cL}{\mathscr{L}}

\newcommand{\R}{\mathbf{R}}
\newcommand{\Z}{\mathbf{Z}}

\newcommand{\frakp}{\mathfrak{p}}

\newcommand{\ad}{\mathrm{ad}}

\newcommand{\into}{\hookrightarrow}

\newcommand{\Y}{\mathscr{Y}}

\newcommand{\la}{\langle}
\newcommand{\ra}{\rangle}

\newcommand{\Nm}{\mathrm{N}}

\newcommand{\rN}{\mathrm{N}}
\newcommand{\rT}{\mathrm{T}}
\newcommand{\rj}{\mathrm{j}}

\newcommand{\al}{\alpha}
\newcommand{\tht}{\theta}
\newcommand{\lm}{\lambda}

\newcommand{\sg}{\sigma}
\newcommand{\Sg}{\Sigma}

\newcommand{\calI}{\mathscr{I}}
\newcommand{\cI}{\mathscr{I}}
\newcommand{{\calG}}{\mathscr{G}}

\newcommand{\cR}{\mathscr{R}}

\newcommand{\calS}{\mathscr{S}}

\newcommand{\cM}{\mathscr{M}}

\newcommand{\cD}{\mathscr{D}}

\newcommand{\bcalS}{\baar{\mathscr{S}}}

\newcommand{\lf}{\ell_{\vphi^{p},\alpha}}

\newcommand{\bC}{\mathbf{C}}

\newcommand{\N}{\mathbf{N}}
\newcommand{\OO}{\mathscr{O}}
\newcommand{\A}{\mathbf{A}}

\newcommand{\bks}{\backslash}
\newcommand{\baar}{\overline}

\newcommand{\cd}{\cdot}

\newcommand{\eps}{\varepsilon}
\newcommand{\vphi}{\varphi}
\newcommand{\vpi}{\varpi}

\newcommand{\Up}{\mathrm{U}}

\newcommand{\wtil}{\widetilde}

\newcommand{\B}{\mathbf{B}}

\newcommand{\vol}{\mathrm{vol}}
\newcommand{\Tr}{\mathrm{Tr}}

\newcommand{\Gal}{\mathrm{Gal}}

\newcommand{\Ker}{\mathrm{Ker}\,}

\newcommand{\Hom}{\mathrm{Hom}\,}
\newcommand{\End}{\mathrm{End}\,}

\newcommand{\llb}{\llbracket}
\newcommand{\rrb}{\rrbracket}

\newcommand{\Spec}{\mathrm{Spec}\,}

\newcommand{\tZ}{\wtil{Z}}

\newsavebox\tempbox
\let\svwidetilde\widetilde
\renewcommand\widetilde[1]{\sbox\tempbox{$#1$}\svwidetilde{\usebox{\tempbox}}}

\def\Xint#1{\mathchoice
      {\XXint\displaystyle\textstyle{#1}}%
      {\XXint\textstyle\scriptstyle{#1}}%
      {\XXint\scriptstyle\scriptscriptstyle{#1}}%
      {\XXint\scriptscriptstyle\scriptscriptstyle{#1}}%
      \!\int}
   \def\XXint#1#2#3{{\setbox0=\hbox{$#1{#2#3}{\int}$}
        \vcenter{\hbox{$#2#3$}}\kern-.5\wd0}}
   
   \def\dashint{\Xint-}

\title[The $p$-adic Gross--Zagier formula at nonsplit primes]{The $p$-adic Gross--Zagier formula on Shimura curves, II:\\nonsplit primes}
\author{Daniel Disegni} 
\address{Department of Mathematics, Ben-Gurion University of the Negev, Be'er Sheva 84105, Israel}

\email{disegni@bgu.ac.il}

\begin{document}
\begin{abstract}
The formula of the title relates $p$-adic heights of Heegner points and derivatives of $p$-adic $L$-functions. It was originally proved by Perrin-Riou for $p$-ordinary elliptic curves over the rationals,  under the assumption that $p$ splits in the relevant quadratic  extension.
We remove  this assumption,  in the more general setting of  Hilbert-modular abelian varieties.
\end{abstract}
\thanks{Research supported by ISF grant 1963/20}

\maketitle
\tableofcontents

\section{Introduction and statement of the main result}
The $p$-adic Gross--Zagier formula of Perrin-Riou relates $p$-adic heights of Heegner points and derivatives of $p$-adic $L$-functions. In its original form \cite{PR}, it concerns (modular) elliptic curves over $\Q$, and it is proved under two main assumptions: first, that the elliptic curve  is $p$-ordinary; second, that $p$ splits in the field $E$ of complex multiplications of the Heegner points. The  formula has  applications to both the  $p$-adic and the classical Birch and Swinnerton-Dyer conjecture.

{\begin{figure}[h]
$$\includegraphics[width=.37\textwidth]{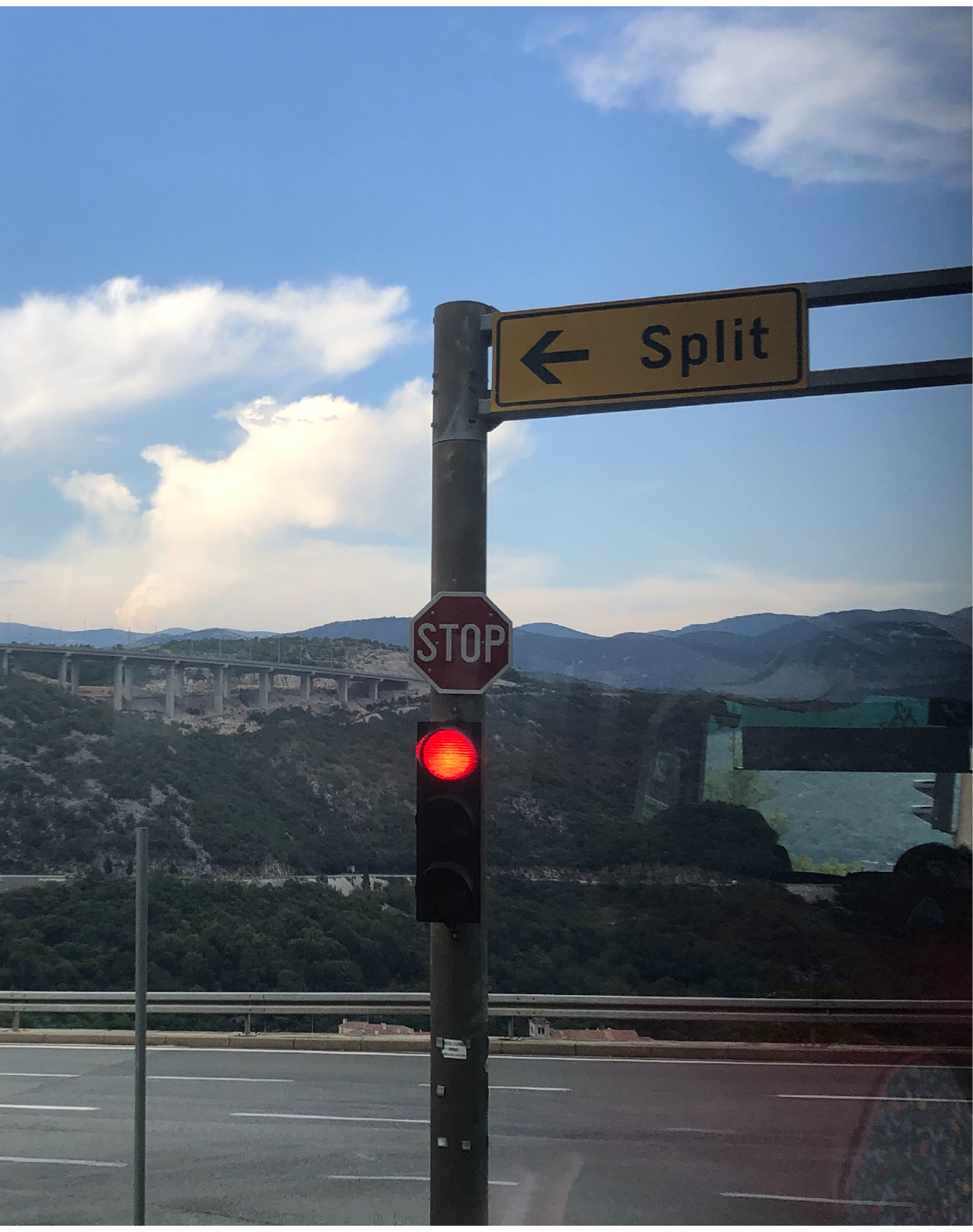} 
$$
\caption{A road sign in Croatia.}
\end{figure} }

The first assumption was removed by Kobayashi \cite{kobayashi} (see also \cite{BPS}). The purpose of this work is to remove the second assumption.

We work in the context of \cite{1}, which will enable us in \cite{univ} to deduce, from the  formula presented here, the analogous one for   higher-weight (Hilbert-) modular motives, as well as a version in the universal ordinary family with some new applications. Nevertheless, the new idea  we introduce  is essentially orthogonal to previous innovations, including those of \cite{1} (and in fact it can be applied, at least in principle, to the non-ordinary case as well). For this reason 
we start in \S~\ref{11} by informally discussing it in the simplest classical case of elliptic curves over $\Q$.
The general form of our results is presented in \S~\ref{statement}.

\subsection{The main ideas in a classical context} \lb{11}
Classically, Heegner points on the elliptic curve $A_{/\Q}$  are images of CM points (or divisors) on a modular curve $X$, under a parametrisation $f\colon X\to A$. More precisely, choosing an imaginary quadratic field $E$, for each ring class character $\chi\colon \Gal(\baar{E}/E)\to \baar{\Q}^{\ts}$, one can construct  a point $P(f, \chi)\in A_{E}(\chi)$, the $\chi$-isotypic part of $A(\baar{E})_{\baar \Q}$.
The landmark formula of Gross--Zagier \cite{gz} relates the height of  $P(f,\chi)$ to the derivative $L'(A_{E}\ot \chi, 1)$ of the $L$-function of a twisted base-change of $A$.  The analogous formula in $p$-adic coefficients\footnote{We denote by `$\doteq$' equality up to a less important nonzero factor.}
\beq
\lb{pr f}\la P(f,\chi), P(f,{\chi}^{-1})\ra\doteq {d\over ds}_{|s=0} L_{p}(A_{E}, \chi\cd\chi_{\rm cyc}^{s})
\eeq
relates cyclotomic derivatives of  $p$-adic $L$-functions to $p$-adic height pairings $\la \ , \ \ra$.   We outline its proof for $p$-ordinary elliptic curves.

\paragraph{Review of Perrin-Riou's proof} Our basic strategy is still Perrin-Riou's variant of  the one of Gross--Zagier; we briefly and informally review it, ignoring, for simplicity of exposition, the role of the character~$\chi$.
Throughout the following discussion, we include pointers to  corresponding statements in the main body of the paper, as guideposts meant to assist the reader's navigation through the more technical framework used there.

Denoting by $\vphi$  the ordinary eigenform attached to $A$, each side of \eqref{pr f} is expressed as the image under a functional  ``$p$-adic Petersson product with $\vphi$'', denoted by $\ell_{\vphi} =\eqref{lfeq}$, of a certain kernel function (a  $p$-adic modular form).  

For the left-hand side of \eqref{pr f}, the form in question is the generating series (cf. \eqref{the series})
\beq \lb {i d z} Z = \sum_{m\geq 1} \la P^{0}, T_{m}P^{0}\ra_{X} {\bf q}^{m} =\sum_{v} Z_{v},\eeq
where $P^{0}\in \Div^{0}(X)$ is a degree-zero modification  of the CM point $P\in X$,   $\la \ , \ \ra_{X}$ is a $p$-adic height pairing on $X$ compatible with the one on $A$, and the decomposition \eqref{i d z} into a sum running over all the finite places of $\Q$ (cf. \eqref{eqdec}) follows from a general   decomposition of the global height pairing into a sum of local ones. More precisely, global height pairings  are valued in the completed tensor product $H^{\ts}\bks{H}_{\A^{\infty}}^{\ts}\hat{\ot} L$ of   the finite id\`eles of  the Hilbert class field $H$ of $E$, and of a suitable finite extension $L$ of $\Q_{p}$. The series  $Z_{v}$ collects the local pairings at $w\vert v$, each valued in $H_{w}^{\ts}\hat{\ot}L$.

The analytic kernel $\cI'$ giving the right-hand side of \eqref{pr f} is the derivative of a   $p$-adic family of  mixed theta-Eisenstein series (cf. \eqref{eqderan}). It also enjoys a  decomposition $$\cI'= \sum_{v\neq p} \cI'_{v}$$ where, unlike \eqref{i d z}, the sum  runs over the finite places of $\Q$ \emph{different from $p$} (cf. \eqref{dec an}). Once established that $Z_{v}\doteq \cI_{v}'$ for $v\neq p$ by computations similar to those of Gross--Zagier (cf. Theorem \ref{theo local comp}), it remains to  show that the $p$-adic modular form $Z_{p}$ is annihilated by $\ell_{\vphi}$ (cf. Proposition \ref{kill lp}).

In order to achieve this, one aims at showing that, after acting on $Z_{p}$ by a Hecke operator to replace $P^{0}$ by $P^{[\vphi]}$  (a lift of the component of its image in the $\vphi$-part of ${\rm Jac}(X)$),  the resulting form $Z_{p}^{[\vphi]}$ is \emph{$p$-critical} (cf. Proposition \ref{is crit}). That is, its coefficients 
$$a_{mp^{s}}:=\la P^{[\vphi]}, T_{mp^{s}}P^{0} \ra_{X,p}$$
 decay $p$-adically no slower than   a constant multiple of $p^{s}$.  The $p$-shift of Fourier coefficients  extends the action on modular forms of the operator $\Up_{p}$  -- which in contrast  acts by a $p$-adic unit on the ordinary form $\vphi$: this implies that $p$-critical forms are annihilated by $\ell_{\vphi}$.

To study the terms $a_{mp^{s}}$, one  constructs a sequence of points  $P_{s}\in X_{H}$ whose fields of definitions  are  the layers $H_{s}$ of the anticyclotomic $p^{\infty}$-extension of $E$. The relations they satisfy allow to express 
 $$a_{mp^{s}} = \la P^{[\vphi]}, D_{m,s}\ra_{X,p},$$
 where $D_{m,s}$ is a degree-zero divisor supported at Hecke-translates of $P_{s}$, which are all essentially   CM points of conductor $p^{s}$ defined over $H_{s}$ (cf. Proposition \ref{pseudo}). By a projection formula for $X_{H_{s}}\to X_{H}$, the height  $a_{mp^{s}}$ is then a sum, over primes $w$ of $H$ above $p$, of the images $$N_{s,w}(h_{m,s,w})$$
  of  heights $h_{m,s,w}$ computed  on $X_{H_{s,w}}$, under the norm map $N_{s,w}\colon H_{s, w}^{\ts}\hat{\ot}L \to H_{w}^{\ts} \hat{\ot} L$.  Moreover it can be  shown that the $L$-denominators  of  $h_{m,s,w}\in H_{s, w}^{\ts}\hat{\ot}L$  are uniformly  bounded (cf. Proposition \ref{di bounded}), so that here we may ignore them and think of $h_{m,s,w}\in H_{s, w}^{\ts}\hat{\ot}\OO_{L}$.
 
A simple observation from  \cite{1} is  that the valuation $w(h_{m,s,w})$ equals 
\beq\lb{int int}
m_{X_{H_{w}}}( P^{[\vphi]} ,D_{m,s}),
\eeq 
the intersection multiplicity of the \emph{flat extensions} (\S~\ref{flex})  of those divisors to some regular integral model  $\X$ of $X_{H_{w}}$. In the split case, it is almost immediate to see that this intersection multiplicity vanishes. This implies that 
\beq\lb{nos}
N_{s,w}(h_{m,s} )\in N_{s,w}(\OO_{H_{s,w}}^{\ts})\hat{\ot}\OO_{L}\subset H_{w}^{\ts}\hat{\ot}\OO_{L}.\eeq 
Since the extension $H_{s,w}/H_{w}$ is totally ramified of degree $p^{s}$, the subset in \eqref{nos}  is  $  p^{s} (\OO_{H_{w}}^{\ts}\hat{\ot}\OO_{L})\subset p^{s} (H_{w}^{\ts}\hat{\ot}\OO_{L}) $, as desired.

\paragraph{The nonsplit case} In the nonsplit case, the $p$-adic intersection multiplicity  has no reason to vanish. However, the  above argument  will still go through if we more modestly show that  \eqref{int int} itself decays at least  like a multiple of $p^{s}$ (cf. Lemma \ref{val cool}). The idea to prove  this  is very simple: we show that if $s$ is large then, for the purposes of computing  intersection multiplicities with other divisors $\cD$ on $\X$, the Zariski closure  of a CM point of conductor at least $p^{s}$  can \emph{almost} be approximated by some  irreducible component  $V$ of the special fibre of $\X$; hence the multiplicity will be zero if $\cD$ arises as a flat extension of its generic fibre.  The  qualifier `almost' means  that the above holds  \emph{except} if $|\cD|$ contains  $V$ itself, which will  be responsible for a multiplicity error term  equal to a constant multiple of $ p^{s}$.

 The approximation result, Proposition \ref{approx}, is precisely   formulated in   an (ultra)metric space of irreducible divisors on the local ring of a regular arithmetic surface, which we introduce following a recent work of  Garc\'ia Barroso, Gonz\'alez P\'erez and Popescu-Pampu \cite{ultrametric}.  The proof of the result is also rather simple (albeit not effective), relying on Gross's theory of quasicanonical liftings \cite{gross}.  The problem of effectively identifying the  approximating divisor $V$ is treated in \cite{equi}.

\paragraph{Subtleties} The above description ignores several difficulties of a relatively more technical nature, most of which we deal with by the representation-theoretic approach of \cite{1} (in turn adapted from Yuan--Zhang--Zhang \cite{yzz}). Namely, we allow for arbitrary modular parametrisations $f$, resulting into an extra parameter $\phi$ in the kernels $Z$ and $\cI'$. By representation-theoretic results, one is free to some extent to \emph{choose} the  parameter $\phi$ to work with  without losing generality. A fine choice (or rather a pair of choices) for its $p$-adic component is dictated by the goal of interpolation, while imposing suitable conditions on its other components allows to circumvent many obstacles in the proof.

\subsection{Statement}\lb{statement}  We now describe our result in the general context  in which we prove it -- which is the same as that of \cite{1} (and \cite{yzz}), to which we refer for a less terse  discussion of the background. (At some points, we find some   slightly different  formulations or normalisations from  those of   \cite{1} to be more natural: see \S~\ref{sec:rev} for the equivalence.)

\paragraph{Abelian varieties parametrised by Shimura curves}
Let $F$ be a totally real field and let $A_{/F}$ be a simple abelian variety of $\GL_{2}$-type. Assume that $L(A,s)$ is modular (this is known in many case if $A$ is an elliptic curve).
Let $\B$ be a quaternion algebra over the ad\`eles $\A=\A_{F}$ of $F$, whose ramification set $\Sg_{\B}$ has odd cardinality and contains all the infinite places. To $\B$ is attached a tower of Shimura curves $(X_{U/F})_{U\subset \B^{\infty\ts}}$, with respective  Albanese varieties $J_{U}$. It carries a canonical system of divisor classes $\xi_{U}\in {\rm Cl}(X_{U})_{\Q}$ of degree $1$, providing a system $\iota_{\xi}$ of maps $\iota_{\xi, U}\in \Hom_{F}(X_{U}, J_{U})_{\Q}$ defined by $P\mapsto  P - \deg(P) \xi_{U}$.

 The space 
$$\pi=\pi_{A, \B}=\varinjlim_{U} \Hom^{0}(J_{U}, A)$$
is either zero or a smooth irreducible representation of $\B^{\ts}$ (trivial at the infinite places), with coefficients in the number field $M:=\End^{0}(A)$. We assume we are in the case $\pi=\pi_{A , \B}\neq 0$, which under the modularity assumption and the condition \eqref{root n} below can be arranged by suitably choosing $\B$. Then  for all places $v\nmid \infty$,  $L_{v}(A, s)=L_{v}(s-1/2 , \pi)$ in $M\ot \bC$.
We denote by $\omega\colon F^{\ts}\bks \A^{\ts} \to M^{\ts}$ the central character of $\pi$.
We have a canonical isomorphism $\pi_{A^{\vee}, \B}\cong \pi_{A,\B}^{\vee}$, see \cite[\S 1.2.2]{yzz}, and we denote by $( \ , \ )_{\pi}\colon \pi_{A, \B}\ot \pi_{A^{\vee}, \B}\to M $ the duality pairing.   

\paragraph{Heegner points}  Let $E/F$ be a CM quadratic extension with associated quadratic character $\eta$, and assume that $E$ admits  an $\A$-embedding $E_{\A^{}}\into \B^{}$, which we fix. Then $E^{\ts }$ acts on the right on $X=\varprojlim_{U}X_{U}$. The fixed-points subscheme $X^{E^{\ts}}\subset X$ is $F$-isomorphic to  $\Spec E^{\rm ab}$, and we fix a point $P\in X^{E^{\ts}}(E^{\rm ab})$. 
Let $$\chi\colon E^{\ts}\bks E_{\A^{\infty}}^{\ts} \cong \Gal(E^{\rm ab}/E) \to L(\chi)^{\ts}$$ be a character valued in a field extension of $L(\chi)\supset M$, satisfying 
$$\omega\cd\chi_{|\A^{\infty\ts}}=\one, $$
and let $$A_{E}(\chi):=(A(E^{\rm ab})\otimes_{M}L(\chi)_{\chi})^{\Gal(E^{\rm ab}/E)},$$ where $L(\chi)_{\chi}$ is an $L(\chi)$-line with Galois action by $\chi$.

Then we have a \emph{Heegner point functional} 
\beq 
f\quad\mapsto\quad  P(f, \chi):=\dashint_{\Gal(E^{\rm ab}/E)}f(\iota_{\xi}(P)^{\tau})\otimes\chi(\tau)\, d\tau\ \   \in A_{E}(\chi)\eeq
(integration for the Haar measure of volume $1$)  in the space of invariant linear functionals 
$$\H(\pi_{A, \B},\chi)\ot_{L(\chi)}A_{E}(\chi) ,\qquad \H(\pi,\chi):= \Hom_{E_{\A^{}}^{\ts}}(\pi\ot\chi, L(\chi))$$
where $E_{\A^{}}^{\ts}$ acts diagonally. There is a product decomposition $\H(\pi, \chi)=\bigotimes_{v}\H(\pi_{v}, \chi_{v})$, where similarly 
$\H(\pi_{v}, \chi_{v}):=\Hom_{E_{v}^{\ts}}(\pi_{v}\ot\chi_{v}, L(\chi))$.
\paragraph{A local unit of measure for invariant functionals} By foundational local results of Waldspurger, Tunnell, and Saito,   the dimension of $\H(\pi,\chi)$ (for any representation $\pi $ of $\B^{\ts}$) is either  $0$ or $1$.  If $A$ is modular and  the global root number
\beq\lb{root n}
\eps(A_{E}\ot\chi)=-1\eeq
then the set of local root numbers determines  a unique quaternion algebra $\B$ over $\A$, satisfying the conditions required above and containing $E_{\A}$, such that $\pi_{A, \B}\neq 0$ and   $\dim_{L(\chi)} \H(\pi_{A, \B},\chi) =1$.\footnote{If $\eps(A_{E}\ot\chi)=+1$ there is no such quaternion algebra and all Heegner points automatically vanish.} 
 We place ourselves in this case; then  there is a canonical factorisable generator 
$$Q_{(, ), dt} =\prod_{v}Q_{(, )_{v}, dt_{v}} \in \H(\pi,\chi) \ot \H(\pi^{\vee},\chi^{-1})$$
depending on the choice of a pairing $(,)=\prod_{v}(, )_{v}\colon \pi\ot \pi^{\vee}\to L(\chi)$ and a measure $dt=\prod_{v} dt_{v}$ on 
$E_{\A^{}}^{\ts}/\A^{\ts}$. It is defined locally as follows. Let us use symbols $V_{(A, \chi)}$ and $V_{(A, \chi), v}$, which we informally think of as denoting (up to abelian factors) the `virtual motive over $F$ with coefficients in $L(\chi)$'
$$\text{``$V_{(A, \chi)}={\rm Res}_{E/F}(h_{1}(A_{E})\ot\chi)\ominus \ad(h_{1}(A)(1))$''}$$ 
 and  its local components (the associated local Galois representation or, if $v$ is archimedean,  Hodge structure). Then we  let, for each place $v$ of $F$ and any auxiliary $\iota\colon L(\chi)\into \bC$,\footnote{Explicitly, if $v$ is archimedean  we have $\cL(V_{(A, \chi), v}, 0)=2$ and $Q_{(,)_{v}, dt_{v}}(f_{1,v}, f_{2,v})= 2^{-1}\vol(\bC^{\ts}/\R^{\ts}, dt_{v}) (f_{1,v}, f_{2,v})$.}
\begin{align}
\lb{calLv}
\cL(\iota V_{(A, \chi), v},s)&:= {\zeta_{F,v}(2) L(1/2+s,\iota \pi_{E,v}\ot \iota\chi_{v}) \over {L(1, \eta_{v})L(1,\iota \pi_{v}, \ad)}}
\cdot  \begin{cases} 1 \quad &\text{if $v$ is finite}\\ \pi^{-1}  &\text{if
 $v\vert \infty$}\end{cases}
\quad \in \iota L(\chi),\\
\lb{Qv1}
 Q_{(, )_{v}, dt_{v}}(f_{1, v}, f_{2,v}, \chi_{v})&:=
\iota^{-1}\cL(\iota V_{(A, \chi),v},0)^{-1}\
 \int_{E_{v}^{\ts}/F_{v}^{\ts}}\chi(t_{v}) (\pi_{v}(t_{v})f_{1,v}, f_{2, v})_{v}\, dt_{v} .
\end{align}

We make the situation more canonical by choosing $dt=\prod_{v} dt_{v} $ to satisfy 
$$\vol(E^{\ts}\bks E_{\A}^{\ts}/\A^{\ts}, dt)=1$$
and by defining, for any $f_{3}\in \pi$, $f_{4}\in \pi^{\vee}$ such that $(f_{3}, f_{4})\neq 0$,
\beq\label{Qv intro} 
Q\left( {  f_{1} \ot  f_{2}\over f_{3}\ot f_{4}};\chi\right) := {Q_{( , ), dt}(f_{1}, f_{2},\chi)\over ( f_{3}, f_{4})}.
\eeq

\paragraph{$p$-adic heights}  Let us fix a prime $\frakp$ of $M$ and denote by $p$ the underlying rational  prime. Suppose from now on that for each $v\vert p$,  
$A_{F_{v}}$ has  $\frakp$-ordinary (potentially good or semistable) reduction. That is, that for a sufficiently large finite extension $L\supset M_{\frakp}$, the 
 rational $\frakp$-Tate module $W_{v}:=V_{p}A\ot_{M} L$
 is a   reducible   $2$-dimensional representation of $\Gal(\baar{F}_{v}/F_{v})$:
\beq\lb{splitting}
0 \to W_{v}^{+} \to W_{v}\to W_{v}^{-} \to 0.\eeq
Fix  such a coefficient field $L$, 
 and for  each $v\vert p$ let
  $\alpha_{v}\colon F_{v}^{\ts}\cong \Gal(F_{v}^{\rm ab}/F_{v})\to L^{\ts}$ be the character giving the action on the  twist $W_{v}^{+}(-1) $. The field $L(\chi)$ considered above will from now on be assumed to be an extension of $L$.
Under those conditions there is a canonical $p$-adic \emph{height pairing} 
$$\la \ , \ \ra\colon A_{E}(\chi) \ot A_{E}^{\vee}(\chi^{-1}) \to \Gamma_{F}\hat{\ot}L(\chi), $$
where $\Gamma_{F}:= \A^{\ts}/\baar{F^{\ts}\widehat{\OO}^{p , \ts}_{F}}$  (the bar denotes Zariski closure). It is normalised `over $F$' as in \cite[\S4.1]{1}. 

For $f_{1}, f_{3}\in \pi$, $f_{2}, f_{4}\in \pi^{\vee}$, and $P^{\vee} \colon \pi^{\vee}\ot \chi\to A_{E}^{\vee}(\chi^{-1})$ the  Heegner point functional of  the dual, our result will measure the ratio $\la P(f_{1}, \chi), P^{\vee}(f_{2}, \chi^{-1})\ra/( f_{3}, f_{4})_{\pi} $,  against the value at the $f_{i}$ of the `unit' $Q$.  The size will be given by the derivative of the  $p$-adic $L$-function that we now define.

\paragraph{The $p$-adic $L$-function} 
We continue to assume that $A$ is $\frakp$-ordinary, and review the definition of the $p$-adic $L$-function from \cite[Theorem A]{1} (in an equivalent form). We start by defining the space on which it lives. Write $\Gamma_{F}=\varprojlim_{n}\Gamma_{F,n}$ as the limit of an inverse system of finite groups, and let 
\beq \lb{YF}
\Y_{F}^{\rm l.c.} :=\bigcup_{n}\Spec L[\Gamma_{F,n}] \quad \subset \quad 
\Y_{F} := \Spec \OO_{L}\llb \Gamma_{F}\rrb\ot_{\OO_{L}}L.\eeq
Then $\Y_{F}$ is a space of continuous characters on $\Gamma_{F}$, and the $0$-dimensional ind-scheme $\Y_{F}^{\rm l.c.}$ is its subspace of locally constant (finite-order) characters.

For a character   $\chi' \colon  \Gal(E^{\rm ab}/E)\to L'^{\ts}$ together with an embedding $\iota \colon L'\to  \bC$,
we shall interpolate the ratio of  complete $L$-functions
$$\cL( \iota V_{(A, \chi')}, s):= \prod_{v} \cL( \iota V_{(A, \chi'),v}, s),\qquad  \cL(\iota V_{(A, \chi'),v}, s)=\eqref{calLv}, \qquad \Re(s)\gg0$$
where the product runs over all places of $F$.

We now define the $p$-interpolation factors for the $p$-adic $L$-function. First, recall that the (inverse) Deligne--Langlands gamma factor of a  Weil--Deligne representation $W'$ of $\Gal(\baar{F}_{v}/F_{v})$ over a $p$-adic field $L'$, with respect to a nontrivial character $\psi_{v}\colon F_{v}\to\bC^{\ts}$ and  an embedding $\iota\colon L'\into \bC$,  is defined as\footnote{The terms $L$ and $\eps$ are normalised as in \cite{tate-nt}.}
$$\gamma(W',\psi_{v})^{-1}:= {L(W')\over \eps(W', \psi_{v}) L(W^{*}(1))}. $$
 Let $\psi=\prod_{v}\psi_{v}\colon F\bks \A \to \bC^{\ts}$ be the standard additive character such that $\psi_{\infty}(\cd) = e^{2\pi i \Tr_{F_{\infty}/\R}(\cd)}$; let $\psi_{E}=\prod_{w}\psi_{E, w}=\psi\circ \Tr_{\A_{E}/\A}$. For  a place $v\vert p $ of $F$, let $d_{v}$ be a generator of the different ideal of $F_{v}$. For a character $\chi'\colon \Gal(\baar{E}/E)\to \bC^{\ts}$, we  define
\beq \lb{eVv}
e_{v}(V_{(A, \chi')})= |d_{v}|^{-1/2}
 {\prod_{w\vert v}\gamma(\iota{\rm WD}(W^{+}_{v|\calG_{E, w}} \ot \chi'_{w}), \psi_{E,w})^{-1}
 \over \gamma(\iota{\rm WD}(\ad(W_{v})(1))^{++}, \psi_{v})^{-1}}
 \cd \cL( V_{(A, \chi'),v})^{-1},\eeq
 where  where $ \ad(W_{v})(1)^{++}:= \Hom(W^{-}_{v}, W^{+}_{ v})(1)= \omega_{v}^{-1}\al_{v}^{2}|\ |_{v}^{2}$,
 and $\iota {\rm WD}$ is the functor from potentially semistable Galois representations to complex Weil--Deligne representations of \cite{fontaine}.

\begin{theoA}\lb{A}
There is a function 
$$\cL_{p}(V_{(A, \chi)})\in \OO(\Y_{F})$$
characterised by the following property.  For  each complex geometric point $s=\chi_{F}\in \Y_{F}^{\rm l.c.}(\bC) $, with underlying embedding $\iota\colon L(\chi_{F})\into \bC$, 
$$\cL_{p}(V_{(A, \chi)},s)=  \iota {e}_{p}({ V}_{(A, \chi')}) \cd \cL(\iota V_{(A, \chi')}, 0), \qquad\chi':= \chi\cd \chi_{F|\Gal(\baar{E}/E)},$$
where $  {e}_{p}(V_{(A, \chi')}):= \prod_{v\vert p} e_{v}(V_{(A, \chi')})$. 
\end{theoA}
The factor $ {e}_{p}(V_{(A, \chi')})$ coincides with the one predicted by Coates and Perrin-Riou (see \cite{coates}) for $V_{(A, \chi')}$ (their conjecture motivates the denominator terms in \eqref{eVv}, which are constants), up  to  the removal of a trivial zero from their interpolation factor for $\ad(W_{v})(1)$.

 \paragraph{The $p$-adic Gross--Zagier formula}
We are almost ready to state our main result. 
Denote by $0\in \Y_{F}$ the point corresponding to $\chi_{F}=\one$, and 
let $$\cL_{p}'(V_{(A, \chi)}, 0):= {\rm d}\, \cL_{p}(V_{(A, \chi)}, 0)\in T_{0}\Y_{F}\cong \Gamma_{F}\hat{\ot} L(\chi).$$

  We say that $\chi_{p}$ is \emph{sufficiently ramified} if it is nontrivial  on a certain open subgroup of $\OO_{E,p}^{\ts}$ depending only on $\omega_{p}$  (see Assumption \ref{assp1} below for the precise definition and a comment).

\begin{theoA} \lb{MT}
Suppose that  the abelian variety $A_{/F}$ is modular and that for all $v\vert p$,  the $\Gal(\baar{F}_{v}/F_{v})$-representation $V_{\frakp}A$ is ordinary and potentially crystalline. Let $\chi\colon \Gal(E^{\rm ab}/E)\to L(\chi)^{\ts}$ be a finite-order character satisfying $$\eps(A_{E}\ot\chi)=-1,$$  and suppose that $\chi_{p}$ 
 sufficiently ramified.

Then for any $f_{1},f_{3}\in \pi$, $f_{2},f_{4}\in \pi^{\vee}$ such that $(f_{3}, f_{4})_{\pi}\neq 0$, we have
$$
 {\la P(f_{1}, \chi), P^{\vee}(f_{2}, \chi^{-1})\ra
\over 
( f_{3}, f_{4})_{\pi} }
=
{e}_{p}(V_{(A, \chi)})^{-1}\cd 
 \cL_{p}'(V_{(\A , \chi)}, 0) \cd 
Q\left( {  f_{1} \ot  f_{2}\over f_{3}\ot f_{4}} ; \chi\right) 
$$
in $ \Gamma_{F}\hat{\ot} L(\chi)$.
\end{theoA}

\begin{rema} \lb{rmk remove} The technical assumptions that $\chi_{p}$ is sufficiently ramified and that $V_{\frakp}A$ is potentially crystalline\footnote{An assumption of this sort is equally necessary in the proof of the main theorem of \cite{1}, see Appendix \ref{app err}.}
 are removed by $p$-adic analytic continuation in \cite[Theorem B]{univ}, and replaced by the (necessary) assumption that that $\chi_{p}$ is \emph{not exceptional} for $A$, that is ${e}_{p}(V_{(A, \chi)})\neq 0$ (which in our case is implied by the potential crystallinity).

Note that for the removal of the former assumption, one only needs the anticyclotomic formula analogous to \cite[Theorem C.4]{1}, and not the full generality of  the multivariable formula in \cite[Theorem D]{univ}.
 \end{rema}
\begin{rema} Concrete versions of the formula of Theorem \ref{MT} may be obtained by choosing explicit  parametrisations $f_{i}$ and evaluating the term $Q$. This is a local problem, solved in \cite{CST}. In particular, by starting from Theorem \ref{MT} (as generalised to all characters following Remark \ref{rmk remove}) and applying the same steps as in the proofs of \cite[Theorems 4.3.1, 4.3.3]{exc}, we obtain the simple  $p$-adic Gross--Zagier formula in anticyclotomic families for elliptic curves $A_{/\Q}$ proposed in  \cite[Conjecture 4.3.2]{exc}, and similarly the direct analogue\footnote{Of course, this is a long detour to get there;  readers interested exclusively in the removal of the `$p$ splits' assumption from Perrin-Riou's formula, or from its analogue over totally real fields, may prefer to try and insert the new argument of the present paper into Perrin-Riou's proof,  or respectively into \cite{dd-ant}.}
 of Perrin-Riou's original result in \cite{PR}. 
\end{rema}

 The theorem has familiar applications extending to the nonsplit case those from \cite{1} (when the other ingredients are available); we leave their formulation to the interested reader, and highlight   instead an application specific to this case pointed out in \cite{exc}, as well as a new application to the non-vanishing conjecture for  $p$-adic heights.
 
\paragraph{A new proof  of a result of Greenberg--Stevens}  
As noted in \cite[Remark 5.2.3]{exc}, the anticylcotomic formula indicated in the previous remark, combined with a result of Bertolini--Darmon, gives yet another proof (quite likely the most complicated so far, but amenable to generalisations) of the following famous result of Greenberg and Stevens \cite{GS}. If $A_{/\Q}$ is an elliptic curve of split multiplicative reduction at $p$, with N\'eron period $\Omega_{A}$ and $p$-adic $L$-function $L_{p}(A, -)$ on $\Y_{\Q}$, then $L_{p}(A, \one)=0$ and 
\beq\lb{eqMT}
 L_{p}'(A,\one) =
\lambda_{p}(A)
 \cd {L(A, 1)\over \Omega_{A}},\eeq
where
 $\lambda_{p}(A)$
  is the $\mathcal{L}$-invariant of Mazur--Tate--Teitelbaum \cite{MTT}.  
  
  We recall a sketch of the argument, referring to \cite{exc} for more details. One chooses an imaginary quadratic   field $E$ such that $p$ is inert in $E$ and that the twist $A^{(E)}$ satisfies  $L(A^{(E)}, 1)\neq 0$. By the anticyclotomic $p$-adic Gross--Zagier formula, $L_{p}'(A_{E},\one)$ is the value at $\chi=\one$ of the height  of an anticyclotomic family $\cP$ of Heegner points.  It is shown in \cite[\S5.2]{bdsurvey}  that the value $\cP(\one)$  equals, in an extended Selmer group,  the Tate parameter $q_{A,p}$ of $A_{\Q_{p}}$ multiplied by a square root of  $L(A_{E}, 1)/\Omega_{A_{E}}$. The height of $q_{A,p}$, in the `extended' sense of \cite{MTT, nek-selmer}, essentially equals $\lm_{p}(A)$. This shows  that, after harmlessly  multiplying by  $L(A^{(E)},1 )/\Omega_{A^{(E)}}$,  the two sides of \eqref{eqMT} are equal.

\subsubsection{Exceptional cases and non-vanishing results} Suppose that $A_{/\Q}$ has multiplicative reduction at a prime $p$ inert in $E$, and that $ L(A_{E}, 1)\neq 0$. Then \emph{for all but finitely many  anticyclotomic  characters $\chi$ of $p$-power conductor, a Heegner point in $A_{E}(\chi)$ is nonzero and the $p$-adic height pairing on $A_{E}(\chi)$ is nondegenerate.} This follows from noting, similarly to the above, that in  the $p$-adic Gross--Zagier formula in anticyclotomic families for $A_{E}$, both sides are nonzero since the heights side specialises, at the character $\chi=\one$, to a nonzero multiple of $\lm_{p}(A)$, which is in turn nonzero by \cite{st et}.

A similar argument, applied to the  formula in Hida families of \cite{univ}, will yield the following result: if $A_{/\Q} $ is an elliptic curve with multiplicative reduction and $L(A,1)\neq 0$, then the Selmer group of the selfdual Hida family ${\bf f}$ through $A$ has generic rank one, and both  the height regulator and the  cyclotomic derivative of the $p$-adic $L$-function of ${\bf f}$ do not vanish. The details will appear in \cite{univ}.

\subsection{Organisation of the paper} 
In \S~\ref{sec2}, we restate our theorems in an equivalent form, direct generalisation of the statements from  \cite{1} (up to a correction involving a factor of $2$, discussed in Appendix B). In \S~\ref{sec3} we recall the proof strategy from \cite{1}, with suitable modifications and corrections. The new argument to treat $p$-adic local heights in the nonsplit case  is  developed in \S~\ref{sec4}. 

We conclude with two appendices, one dedicated to  some local results, the other containing a list of errata to \cite{1}.

\section{Comparison with  \cite{1}}
\lb{sec2}
We  compare Theorems \ref{A} and \ref{MT} with the corresponding results from \cite{1}. We continue with the setup and  notation of \S~\ref{statement}.
\subsection{The $p$-adic $L$-function} We deduce our Theorem \ref{A}  from \cite[Theorem A]{1}.

Let $\sg^{\infty}$ be the nearly $\frakp$-ordinary, $M$-rational  (\cite[Definition 1.2.1]{1}) representation of $\GL_{2}(\A)$ attached to $A$ as in \cite{1}. In Theorem A \emph{ibid.} we have constructed a $p$-adic $L$-function 
$$L_{p, \alpha}(\sg_{E}),$$
which is a bounded function on a  rigid space $\Y'^{\rm rig}_{/L}$   (denoted by $\Y'$ in \cite{1}).  In the construction of \emph{loc. cit.} (and in all this paper), we use the same additive  character $\psi_{p}=\prod_{v\vert p}\psi_{v}$ as in Theorem \ref{A};  see the correction in Appendix \ref{app err} for the exact ring of definition of $L_{p, \al}(\sg_{E})$.

The space $\Y'^{\rm rig}=\Y'^{\rm rig}_{\omega, V^{p}}$   parametrises certain continuous $p$-adic  characters of $E^{\ts}\bks E_{\A}^{\ts}$ invariant under an arbitrarily fixed compact open subgroup $V^{p}\subset E_{\A^{p\infty}}^{\ts}$.
The boundedness means precisely that we may (and do) identify  $L_{p, \al}(\sg_{E})$  with a function on a corresponding scheme
\beq\lb{Y'}
\Y'\subset \Spec \OO_{L}\llb E^{\ts}\bks E_{\A^{\infty}}^{\ts}/V^{p} \rrb\ot {L}, 
\eeq
that, when also viewed as a space of characters $\chi'$, is  the subscheme cut out by the closed condition $\omega\cd \chi'_{|\widehat{\OO}_{F}^{p,\ts}}=\one$.
 Similarly to $\Y_{F}$, the scheme $\Y'$ contains  a $0$-dimensional subscheme $\Y'^{\rm l.c.}$ parametrising the locally constant characters in $\Y'$. The function $L_{p, \al}(\sg_{E})$ is characterised by the following property. 
Denote by $D_{K}$ the discriminant of a number field $K$. Then  at all $\chi'\in \Y'$
 with underlying embedding $\iota\colon L\into\bC$, we have 
\beq\lb{old Lp}
L_{p,\alpha}(\sigma_{E})(\chi')
 =\prod_{v|p}Z_{v}^{\circ}(\chi'_{v}, \psi_{v})\cd
  { \pi^{2[F:\Q]} |D_{F}|^{1/2} \over 2\zeta_{F}(2)} \cd \cL(\iota V_{(A, \chi')})
  \eeq
for certain local factors $Z_{v}^{\circ}$.

Fix a finite-order character $$\chi\colon E^{\ts}\bks E_{\A^{\infty}}^{\ts} \cong \Gal(E^{\rm ab}/E) \to L(\chi)^{\ts}$$ satisfying  $\omega\cd\chi_{|\A^{\infty\ts}}=\one, $
and consider the map 
\beqq
j_{\chi} \colon \Y_{F} &\to \Y'\\
\chi_{F} &\mapsto \chi\cd \chi_{F} \circ N_{E_{\A^{\infty\ts}/\A^{\infty\ts}}}.
\eeqq
\begin{proof}[{Proof of Theorem \ref{A}}]
Let
\beq\lb{cratio}
C(\chi_{p}'):={ e_{p} (V_{(A, \chi')})
\over  \prod_{v\vert p }Z_{v}^{\circ}(\chi'_{v}, \psi_{v})}.\eeq
We show in Proposition \ref{propcratio} that this  is a constant in $C\in L$,  independent of $\chi'_{p}$.

Define 
\beq\lb{lplp}
\cL_{p}(V_{(A, \chi)}):=  {2  \zeta_{F}(2)\over \pi^{2[F:\Q] } |D_{F}|^{1/2}}   \cd  C\cd
 L(1, \sg_{v}, \ad) \cd  j_{\chi}^{*}L_{p, \alpha}(\sg_{E})
\eeq
a function in $\OO(\Y_{F})$. 
It is clear from the definition and \eqref{old Lp} that it satisfies the required interpolation property.
\end{proof}

\subsection{Equivalence of statements}\lb{sec:rev}
We now restate Theorem \ref{MT} in a form that directly generalises \cite[Theorem B]{1}. It is the form in which we will prove it, for convenience of reference.

 We retain the  setup of \S~\ref{statement}. 
 Let ${\rm d}_{F}$ be the   $\Gamma_{F}$-differential  defined before \cite[Theorem B]{1}. For all $v\nmid \infty$ let $dt_{v}$  be the measure on $E_{v}^{\ts}/F_{v}^{\ts}$ specified in  \cite[paragraph after (1.1.2)]{1} if $v\nmid \infty$ and the measure giving $\bC^{\ts}/\R^{\ts}$ volume $2$ if $v\vert \infty$.

\begin{theo}\lb{MT2}
Retain the assumptions of Theorem \ref{MT}, and fix a decomposition $(, )_{\pi}=\prod_{v}( ,)_{v}$, with $(1, 1)_{v}=1$ if $v\vert \infty$. 
Then  for all $f_{1}\in \pi$, $f_{2}\in \pi^{\vee}$,
\beq \lb{eqmt2}
\langle P(f_{1}, \chi), P^{\vee}(f_{2}, \chi^{-1})\rangle = 
{c_{E}}\cdot
\prod_{v|p}Z_{v}^{\circ}(  \chi_{v})^{-1}\cdot {\rm d}_{F} L_{p, \alpha}(\sigma_{A,E})(\chi)\cdot   \prod_{v\nmid \infty} Q_{(, )_{v}, dt_{v}}(f_{1}, f_{2}, \chi)\eeq
in $ \Gamma_{F} \hat{\otimes} L(\chi)$,  where
$$c_{E}:={   \zeta_{F}(2)\over  (\pi/2)^{[F:\Q]}|D_{E}|^{1/2}  L(1, \eta) } \in\Q^{\times}. $$
\end{theo}
\begin{lemm}\lb{is eq} Theorem \ref{MT2}  is equivalent to Theorem \ref{MT}. When every prime $v\vert p$ splits in $E$,  it specialises to \cite[Theorem B]{1}  as  corrected  in Appendix \ref{app err}. 
\end{lemm}
\begin{proof} The second assertion is immediate; we prove the first one. First, we note that \eqref{eqmt2} is equivalent to 
\beq\lb{eqmt2b}
{\langle P(f_{1}, \chi), P^{\vee}(f_{2}, \chi^{-1})\rangle\over (f_{3},f_{4})_{\pi}} = 
{c_{E}}\cdot
\prod_{v|p}Z_{v}^{\circ}(  \chi_{v})^{-1}\cdot {\rm d}_{F} L_{p, \alpha}(\sigma_{A,E})(\chi)
\cdot   2^{-[F:\Q]}\prod_{v} {Q_{(, )_{v}, dt_{v}}(f_{1}, f_{2}, \chi) \over (f_{3,v}, f_{4,v})_{v}}
\eeq
 for any $f_{3}\in \pi$, $f_{4}\in \pi^{\vee}$ with $f_{3, \infty}=f_{4, \infty }=1$ and $(f_{3}, f_{4})_{\pi}\neq 0$ (the extra power of $2$ comes from the archimedean places). The left-hand side of \eqref{eqmt2b} is the same as that of the formula of Theorem \ref{MT}, and the product of the terms after the $L$-derivative in its right-hand side  equals  
$$2^{-[F:\Q]} {\prod_{v} dt_{v}\over dt} \cd  Q\left( {  f_{1} \ot  f_{2}\over f_{3}\ot f_{4}} ; \chi\right) =  2^{1-[F:\Q]} |D_{E/F}|^{1/2}|D_{F}|^{1/2}\pi^{-[F:\Q]}L(1, \eta) \cd Q\left( {  f_{1} \ot  f_{2}\over f_{3}\ot f_{4}} ; \chi\right),$$
because  the measure $\prod_{v} dt_{v}$ (respectively $dt$)  gives  $E^{\ts}\bks \A_{E}^{\ts}/\A^{\ts}$  volume $2 |D_{E/F}|^{1/2}|D_{F}|^{1/2} \pi^{-[F:\Q]}L(1, \eta)$ (respectively $1$).

Next, we have ${\rm d}_{F}L_{p, \alpha}(\sg_{E})(\chi)={1\over 2} {\rm d}( j_{\chi}^{*} L_{p, \alpha}(\sg_{E}) )(\one)$, and it is clear from comparing the interpolation properties that 
$$\prod_{v|p}Z_{v}^{\circ}(  \chi_{v})^{-1} \cd {1\over 2} {\rm d} \left( j_{\chi}^{*} L_{p, \alpha}(\sg_{E}) \right)(\one)= 
{\pi^{2[F:\Q] } |D_{F}|^{1/2} \over 2\zeta_{F}(2)} 
\cd {e}_{p}(V_{(A, \chi)})^{-1}\cd \cL_{p}'(V_{(A, \chi)}, 0).$$
 It follows that the right hand side of \eqref{eqmt2}  equals $c  \cd {e}_{p}(V_{(A, \chi)})^{-1}\cd \cL_{p}'(V_{(A, \chi)}, 0) \cd  Q\left( {  f_{1} \ot  f_{2}\over f_{3}\ot f_{4}} ; \chi\right),$ where 
$$c= {c_{E}} \cd {\pi^{2[F:\Q] } |D_{F}|^{1/2} \over 2\zeta_{F}(2)}  \cd  2^{1-[F:\Q]} |D_{E/F}|^{1/2}|D_{F}|^{1/2} \pi^{-[F:\Q]}L(1, \eta)=1.$$
\end{proof}

\section{Structure of the proof} \lb{sec3}
We review the formal structure of the proof in \cite{1}, dwelling only on those points where the arguments need to be modified or corrected. For an introductory description with some more details than given in \S~\ref{11}, see  \cite[\S 1.7]{1}.
 Readers interested in a  detailed understanding of the present section are advised to keep a copy of \cite{1} handy.

\subsection{Notation and setup}\lb{sec:not}  We very briefly review some notation and definitions  from \cite{1}, which will be used throughout the paper.

\paragraph{Galois groups} If $K$ is a perfect field, we denote by $\calG_{K}:=\Gal(\baar{K}/K)$ its absolute Galois group.

\paragraph{Local fields} For $v$ finite  a place of $F$, we denote by $\vpi_{v}$ a fixed uniformiser and by $q_{F, v}$ the cardinality of the residue field of $F_{v}$. We denote by $d_{v}$ a generator of the absolute different of $F_{v}$, by $D_{v}$ a generator of the relative discriminant of $E_{v}/F_{v}$ (equal to $1$ unless $v$ ramifies in $E$), and by $e_{v}$ the ramification degree of $E_{v}/F_{v}$. If $w\vert v$ is a place of $E$, we denote by $q_{v}\colon E_{v}\to F_{v}$ and $q_{w}\colon E_{w}\to F_{v}$ the relative norm maps.

We denote by 
$\psi=\prod_{v}\psi_{v}\colon F\bks  \A \to \bC^{\ts}$ the additive character fixed before Theorem \ref{A}.

\paragraph{Base-change of rings and schemes} If $R$ is a ring, $R'$ is an $R$-algebra, $M$ is an $R$-module and $S$ is an $R$-scheme, we denote  $M_{R'}=M\ot_{R}R'$, $S_{R'}= S\times_{\Spec R}\Spec R'$.

\paragraph{Groups, measures, integration} We adopt the same notation and  choices of measures as in \cite[\S 1.9]{1},  including a regularised integration $\int^{*}$. In particular $T:= {\rm Res}_{E/F}{\bf G}_{m, E}$, $Z={\bf G}_{m, F}$, and on the adelic points of $T/Z$ we use two measures $dt$ (the same as introduced above Theorem \ref{MT2}) and $d^{\circ}t$. \emph{The measure denoted by $dt$ in the introduction will not be used.}

\paragraph{Operators at $p$} Let $v\vert p$ be a place of $F$. We denote by $\vpi_{v}$ a fixed uniformiser at $v$. For $r\geq 1$ we let $K_{1}^{1}(\vpi_{v}^{r}) \subset \GL_{2}(\OO_{F, v})$ be the subgroup of matrices which become upper unipotent upon reduction modulo $\vpi^{r}$. We denote  by 
$$\Up_{v,*}=K^{1}_{1}(\vpi_{v}^{r}) \twomat 1{}{}{\vpi_{v}^{-1}}K^{1}_{1}(\vpi_{v}^{r}) , \qquad  \Up_{v}^{*}=K^{1}_{1}(\vpi_{v}^{r}) \twomat {\vpi^{r}}{}{}{1}K^{1}_{1}(\vpi_{v}^{r}) ,  $$
the usual double coset operators, and by 
$$w_{r,v}:= \twomat {}1{-\vpi_{v}^{r}}{}\in \GL_{2}(F_{v}).$$
We also let  $w_{r}:=\prod_{v\vert p}w_{r,v}\in \GL_{2}(F_{p})$ and, if $(\beta_{v})_{v\vert p}$ are characters of $F_{v}^{\ts}$, we denote $\beta_{p}(\vpi):=\prod_{v\vert p}\beta_{v}(\vpi_{v})$. 

\paragraph{Spaces of characters} We denote by $\Y_{F}$, $\Y'$, $\Y$ respectively the schemes  over $L$ defined in \eqref{YF}, \eqref{Y'} and the subscheme of $\Y$ cut out by the condition $\chi_{|\A^{\infty\ts}}=\omega^{-1}$. We add to this notation a superscript `l.c.' to denote the ind-subschemes of locally constant characters (which has a model over a finite extension of $M$ in $L$).

Let $\cI_{\Y/\Y'}$ be the ideal sheaf of $\Y\subset \Y'$. If $\cM$ is a coherent  $\OO_{\Y'}$-module, we denote 
$${\rm d}_{F}\colon \cM\ot_{\OO_{\Y'}} \cI_{\Y/\Y'} \to \cM \ot_{\OO_{\Y'}} \cI_{\Y/\Y'}/\cI_{\Y/\Y'}^{2} =\cM_{|\Y}\hat{\ot}\Gamma_{F}$$
the normal derivative (cf. the definition before \cite[Theorems B]{1}.)

\paragraph{Kirillov models} Let $\sg^{\infty}=\bigotimes_{v\nmid \infty}\sg_{v} $ be the $M$-rational  
automorphic representation of $\GL_{2}(\A)$ attached to $A$, and denote abusively still by $\sg^{\infty}$ its base-change to $L$. For every place $v$ the representations $\sg_{v}$ of $\B_{v}^{\ts}$ and $\pi_{v}$ of $\GL_{2}(F_{v})$ are Jacquet--Langlands correspondents. 

 For $v\vert p$, we denote by 
$$\mathscr{K}_{\psi_{v}}\colon \sg_{v} \to C^{\infty}(F_{v}^{\ts}, L)$$
 a fixed rational Kirillov model. 

\paragraph{Orthogonal spaces} We let  ${\bf V}:=\B$ equipped with the reduced norm $q$, a quadratic form valued in $\A$.  The image of $E_{\A}$ is a subspace   ${\bf V}_{1}$ ot the orthogonal space ${\bf V}$,  and we let ${\bf V}_{2}$ be its orthogonal complement. The restriction $q_{|{\bf V}_{1}}$ is the adelisation of the norm of $E/F$.
\paragraph{Schwartz spaces and Weil representation} If $\bf V'$ is any one of the above spaces, we denote by 
${\bcalS}({\bf V}'\ts \A^{\ts})=\bigotimes'_{v} {\bcalS}({\bf V}'_{v}\ts F_{v}^{\ts})$
  the Fock space  of Schwartz functions considered in \cite{1}.  (This differs from the usual Schwartz space only at infinity.) There is  a \emph{Weil representation} 
  $$r=r_{\psi}\colon \GL_{2}(\A)\ts  {\rm O}({\bf V}, q)\to \End {\bcalS}({\bf V}\ts \A^{\ts}),$$
  defined as in \cite[\S 3.1]{1}.  The orthogonal group of $\bf V$ naturally contains the product $T(\A)\ts T(\A)$ acting by left and right multiplication on ${\bf V}$. The Weil representation also depends on a choice of additive characters  $\psi$. The restriction  $r_{|T(\A)\ts T(\A)}$  preserves   the decomposition ${\bf V}_{1}\oplus{ \bf V}_{2}$, hence it accordingly decomposes as $r_{1}\oplus r_{2}$.

\paragraph{Special  data at $p$} We list the functions at the places $v\vert p$ that we use.

Define
\beq \lb{Wv}
W_{v}(y):= \one_{\OO_{F, v}-{0}}(y)|y|_{v} \alpha_{v}(y),\eeq
 the ordinary vector in the fixed Kirillov model $\mathscr{K}_{\psi_{v}}$ of $\sg_{A,v}$. We consider 
\beq\lb{vphip}
\vphi_{v}=\vphi_{v,r}:= \mathscr{K}_{\psi_{v}}^{-1} (\alpha_{v}(\vpi_{v})^{-r}w_{r} W_{v}) \in \sg_{v}.
\eeq

Now we consider Schwartz functions. We let $${\bf B}_{v}\cong M_{2}(F_{v})$$ be the indefinite quaternion algebra over $F_{v}$; this choice is justified \emph{a posteriori} by Corollary \ref{non exc B}.  The following choices  of functions correct and modify the ones  fixed in \cite{1} (cf. the Errata in Appendix \ref{app err}); note in particular that we will use two different functions on ${\bf V}_{2, v}$.  

Decompose orthogonally ${\bf V}_{v}={\bf V}_{1,v}\oplus {\bf V}_{2, v}$, where ${\bf V}_{1,v}=E_{v}$ under the fixed embedding $E_{\A^{\infty}}\into \B^{\infty}$.   We define the following  Schwartz functions on, respectively,  $F_{v}^{\ts}$ and its product with  ${\bf V}_{1,v}$, ${\bf V}_{2,v}$, ${\bf V}_{v}$:
\beq\lb{phip1}
\phi_{F,r}(u)&:=\delta_{1, U_{F, r}}(u),& \text{where }  \delta_{1, U_{F, r}}(u)&:= {\vol(\OO_{F,v}^{\ts})\over \vol(1+\vpi^{r}\OO_{F,v})}\one_{1+\vpi^{r}\OO_{F,v}}(u);\\
\phi_{1, r}(x_{1}, u)& := \delta_{1, U_{T,r}}(x_{1}) \delta_{1, U_{F,r}}(u),& \text{where }  \delta_{1, U_{T,r}}(x_{1}) &= {\vol(\OO_{E,v})\over \vol(1+\vpi^{r}\OO_{E,v})} \one_{1+\vpi^{r}\OO_{E,v}}(x_{1});\\
\eeq
and
\beq \lb{phip}
\phi_{2}^{\circ}(x_{2}, u) & := \one_{\OO_{{\bf V}_{2,v}}}(x_{2}) \one_{\OO_{F,v}^{\ts}}(u);\\
\phi_{2, r}(x_{2}, u)& :=  e_{v}^{-1}|d|_{v}\cd \one_{\OO_{{\bf V}_{2,v}}\cap q^{-1}(-1+\vpi^{r}\OO_{F, v})}(x_{2}) \one_{1+\vpi^{r}\OO_{F,v}}(u); &&\\
\phi_{r}(x, u)&:= \phi_{1, r}(x_{1},u) \phi_{2, r}(x_{2}, u).&&
\eeq

\paragraph{$p$-adic modular forms and $q$-expansions} In \cite[\S2]{1}, we have defined the notion of Hilbert automorphic forms and twisted Hilbert automorphic forms (the latter depend on an extra variable $u \in \A^{\ts}$). We have also defined the associated space of $q$-expansions, and a less redundant space of \emph{reduced} $q$-expansions. When the coefficient field is a finite extension $L$ of $\Q_{p}$ these spaces are endowed with a topology. We have an  (injective) reduced-$q$-expansion map on modular forms, denoted by 
$$\vphi'\mapsto {}^{\bf q}\vphi'.$$
The image of modular forms (respectively cuspforms)  of  level $K^{p}K^1_{1}(p^{\infty})\subset \GL_{2}(\A^{p\infty})$, parallel weight~$2$ and central character $\omega^{-1}$ is denoted by ${\bf M}={\bf M}(K^{p}, \omega^{-1})$ (respectively ${\bf S}$).  The closure of ${\bf M}$ (respectively ${\bf S}$)  in  the space of $q$-expansions with coefficients in $L$ is denoted  ${\bf M}'$, respectively ${\bf S}'$  and its elements are called $p$-adic modular forms (respectively cuspforms).

If $\Y^{?}=\Y_{F}, \Y'$, we define the notion of a $\Y^{?}$-family of modular forms by copying word for word \cite[Definition 2.1.3]{1}; the resulting  notion coincides with that of bounded families on the analogous rigid spaces considered in \emph{loc. cit.}

For a finite set of places $S$ disjoint from those above $p$, we have also defined a certain quotient  space $\baar{\bf S}'_{S}$ of cuspidal reduced $q$-expansions modulo those all of whose coefficients of index $a\in F^{\ts}\A^{S\infty\ts}$ vanish. According to \cite[Lemma 2.1.2]{1}, for any $S$ the reduced-$q$-expansion map induces an \emph{injection}
\beq\lb{bS'}
{\bf S}\into \baar{\bf S}_{S}'.
\eeq

\paragraph{$p$-adic Petersson product and $p$-critical forms}
For $\vphi^{p}\in \sg$, we defined in \cite[Proposition 2.4.4]{1} a functional 
\beq \lb{lfeq}
 \lf\colon {\bf M}(K^{p}, \omega^{-1}, L)\to L, \eeq
whose restriction to classical modular forms  equals, up to an adjoint $L$-value, the limit as $r\to \infty$  of Petersson products with antiholomorphic forms  $\vphi^{p}\vphi_{p,r}\in \sg$ with component $\vphi_{p, r}=\prod_{v\vert p}\vphi_{v, r}$ as in \eqref{vphip}. 

Let $v\vert p$. We say that a form or $q$-expansion over a finite-dimensional $\Q_{p}$-vector space $L$  is \emph{$v$-critical } if its coefficients $a_{*}$ (where $*\in \A^{\infty\ts}$)  satisfy
\beq\lb{defcrit}
a_{m\vpi_{v}^{s}}= O(q_{F, v}^{s})\eeq
in $L$, uniformly in $m\in \A^{\infty\ts}$. Here for two functions $f, g\colon \N\to L$, we write
$$f=O(g)\Longleftrightarrow \text{there is a constant $c>0$ such that } |f(s)|\leq  c|g(s)| \text{\ for all sufficiently large $s$.} $$

The space of \emph{$p$-critical} forms is the sum of the spaces of $v$-critical forms for $v\vert p$. Any element in those spaces is annihilated by $\lf$.

\subsection{Analytic kernel}
The analytic kernel is a  $p$-adic family of theta-Eisenstein series,  related to  the $p$-adic $L$-function. We review its main properties.

 \begin{prop}
There exist  $p$-adic families of $q$-expansions of modular forms $\cE$ over $\Y_{F}$ and $\cI$ over $\Y'$, satisfying: 
\begin{enumerate}
\item  For any $\chi_{F}\in \Y_{F}^{\rm l.c.}(\bC)$  and  any $r=(r_{v})_{v|p}$ satisfying $c(\chi_{F})\vert p^{r}$, 
 we have the identity  of $q$-expansions of  twisted modular forms of weight 1:
$$\mathscr{E}(u,\phi_{2}^{p\infty};\chi_{F})=  
|D_{F}| 
{L^{(p)}(1, \eta\chi_{F})\over L^{(p)}(1, \eta)}
\, {}^{\qqq} E_{r}(u, \phi_{2} ,\chi_{F}),$$
where 
\beq\lb{Er} E_{r}(g,u, \phi_{2} ,\chi_{F}):=\sum_{\gamma\in P^{1}(F)\bks SL_{2}(F)} \delta_{\chi_{F}, r} (\gamma gw_{r}) r(\gamma g) \phi_{2}(0,u)\eeq
is the Eisenstein series defined in \cite[\S 3.2]{1}, with respect to  $\phi_{2}=\phi_{2}^{p\infty}(\chi_{F})\phi^{\circ}_{2,p\infty}$ with $\phi_{2,v}^{\circ}$ as in \eqref{phip} for $v\vert p$, and $\phi_{2, v}(\chi_{v}) $ for $v\nmid \infty $ and   $\phi_{2, v}^{\circ}$ for $v\vert \infty$ as defined in \emph{loc. cit}.
 \item For $\phi_{1}\in \bcalS({\bf V}_{1}\ts {\A}^{\ts})$ and $\chi'\in \Y'^{\rm l.c.}$, consider the twisted modular form of weight 1 with parameter  $t\in E_{\A}^{\ts}$:
\beq \lb{theta def}
\theta(g,(t, 1) u, \phi_{1}):=\sum_{x_{1}\in  E} r_{1}(g, (t, 1))\phi_{1}(x, u).
\eeq
For any $\chi' \in \Y'^{\rm l.c.}$, let $\chi_{F} := \omega^{-1}\chi'_{|\A^{\ts} }\in \Y_{F}^{\rm l.c.}$. Then for any  $r=(r_{v})_{v|p}$ satisfying $r_{v}\geq 1$ and  $c(\chi_{F})\vert p^{r}$ 
we have
\beq\lb{ItE}
\calI(\phi^{p\infty}; \chi')=
{c_{U^{p}} |D_{E}|^{1/2} \over |D_{F}|^{1/2}} \, 
\int_{[T]}^{*}\chi'(t) 
 \sum_{u\in\mu_{U^{p}}^{2}\bks F^{\times}} {}^{\qqq}\theta((t,1),u, \phi_{1}; \chi') \mathscr{E}(q(t)u,\phi_{2}^{p\infty};\chi_{F})\, d^{\circ }t,
\eeq
where for $v\vert p$, $\phi_{1,v}=\phi_{1, v,r}$ is as in \eqref{phip1}.
\item We have
\begin{align}\label{def plf}
\lf(\calI(\phi^{p\infty}))=
L_{p, \alpha}(\sigma_{E})\cd \prod_{v\vert p} |d|^{2}_{v}|D|_{v}
 \prod_{v\nmid p \infty} \mathscr{R}_{v}^{\natural}(W_{v}, \phi_{v}, \chi_{v}')
\end{align}
where the local terms $\cR_{v}^{\natural} $ are as  in  \cite[Propositions 3.5.1, 3.6.1]{1}.
\end{enumerate}
\end{prop}
\begin{proof} Part 1 is \cite[Proposition 3.3.2]{1}. Part 2 summarises \cite[\S 3.4]{1}. Part 3 is \cite[(3.7.1)]{1} with the correction of  Appendix B below.
\end{proof}

\paragraph{Derivative of the analytic kernel}
We denote 
\beq \lb{eqderan}
\cI'(\phi^{p\infty};\chi):= {\rm d}_{F}\cI(\phi^{p\infty};\chi), \eeq
  a $p$-adic modular form with coefficients in $\Gamma_{F}\hat{\ot} L(\chi)$.

\subsection{Geometric kernel} The geometric kernel function, see \cite[\S\S5.2-5.3]{1}, is related to the  heights of Heegner points. We recall its construction and modularity. 

\subsubsection{CM divisors} For any $x\in \B^{\infty \ts}$, we have a Hecke translation ${\rm T}_{x}\colon X\to X$, and a Hecke correspondence $Z(x)_{U}$ on $X_{U }\ts X_{U}$.
Fix any $P\in X^{E^{\ts}}(E^{\rm ab})$, and for $x\in \B^{\infty\ts}$, let $[x]:={\rm T}_{x}P$ be the Hecke-translate of $P$ by $x$, and let $[x]_{U}$ be its image in $X_{U}$.  If $H/E$ is any finite extension,  the points in $X_{U, H}$ corresponding to Galois orbits of points of the form $[x]_{U}$ are called \emph{CM points} (for the CM field $E$).

Let  ${\rm Cl}(X_{U,{\baar{F}}})_{\Q}\supset {\rm Cl}^{0}(X_{U,{\baar{F}}})_{\Q}$ be the space of divisor classes with $\Q$-coefficients and, respectively, its subspace consisting of classes with degree $0$ on  every connected component. Denote by   ${(\ )}^{0}\colon {\rm Cl}(X_{U,{\baar{F}}})_{\Q}\to{\rm Cl}^{0}(X_{U,{\baar{F}}})_{\Q}$  the linear section of the inclusion  whose kernel is spanned by  the pushforwards to $X_{U, \baar{F}}$ of the  classes  of the canonical bundles  of the connected components of   $X_{U', \baar{F}} $, for any sufficiently small $U'$. 

We define the $\chi$-isotypic CM divisors  
\beqq 
t_{\chi} &:=\int^{*}_{[T]}\chi(t) [t^{-1}]^{}_{U}d^{\circ}t   \in \Div(X_{U, \baar{F}})_{L(\chi)}, \\
t_{\chi}^{0} &:=\int^{*}_{[T]}\chi(t) [t^{-1}]^{0}_{U}d^{\circ}t   \in \Div^{0}(X_{U, \baar{F}})_{L(\chi)} ,\eeqq
where  the integrations simply reduce to (normalised) finite sums.

\subsubsection{Generating series} For $a\in \A^{\infty\ts}$, $\phi^{\infty}\in \bcalS({\bf V}\ts \A^{\ts})$,  consider the correspondences 
\begin{equation}\label{za}
{\wtil Z}_{a}(\phi^{\infty}):= c_{U^{p}}w_{U} |a| \sum_{x\in U \bks \B^{\infty\times}/ U }\phi^{\infty}(x, aq(x)^{-1}) Z(x)_{U}
\end{equation}
where $w_{U}=|\{\pm 1\}\cap U| $ and $c_{U^{p}}$ is defined in  \cite[(3.4.3)]{1}.  By \cite[Theorem 5.2.1]{1} (due to Yuan--Zhang--Zhang),
there is an automorphic form  
\beq\lb{zphi}
{\wtil Z}(\phi^{\infty}) \quad \in \quad  C^{\infty}(\GL_{2}(F)\bks \GL_{2}(\A), \bC)\ot_{\Q} {\rm Pic}(X_{U}\ts X_{U})_{\Q},\eeq whose $a^{\rm th}$ reduced coefficient is the image of  ${\wtil Z}_{a}(\phi^{\infty})$ for each $a\in \A^{\infty\ts}$.

Let
\beq
\lb{htJ}
\la\ , \ \ra=\la\ , \ \ra_{X}\colon  J^{\vee}(\baar{F}) \ts J(\baar{F}) \to \Gamma_{F}\hat{\ot} L
\eeq
be the $p$-adic height pairing  defined as in\footnote{There is a typo in \emph{loc. cit.} (also noted in Appendix \ref{app err} below):  the left-hand side of the last equation in the statement should be $\la f_{1}'(P_{1}), f_{2}'(P_{2})\ra_{J, *}$.} \cite[Lemma 5.3.1]{1}. (We abusively omit the subscript $X$ as we will no longer need to use the pairing on $A_{E}(\chi)\ot A_{E}^{\vee}(\chi)$.)

We define the \emph{geometric kernel} to be
\beq\lb{the series}
{\wtil Z}(\phi^{\infty},\chi):=
\sum_{a\in F^{\ts}} \la {\wtil Z}_{a}(\phi^{\infty})[1]_{U}^{0}, t^{0}_{\chi} \ra\, {\bf q}^{a}.
\eeq
By [I, Proposition 5.3.2 and  formula following its proof],  the series 
$ {\wtil Z}(\phi^{\infty},\chi)$
is (the $q$-expansion of) a weight-$2$ cuspidal Hilbert modular form  of central character $\omega^{-1}$, with coefficients in $\Gamma_{F}\hat{\ot} L(\chi)$.

\subsubsection{Geometric kernel and Shimizu lifts}
Let
$$\theta_{\iota_{\frakp}}\colon( \sigma^{\infty}\otimes \bcalS({\bf V}^{\infty}\times \A^{\infty,\times}))\otimes_{M}L\to (\pi\otimes \pi^{\vee})\otimes_{M, \iota_{\frakp}} L$$
be   Shimizu's theta lifting  defined in \cite[\S 5.1]{1}.  Let 
 $${\rm T}_{\rm alg}\colon \pi^{U}\otimes_{M} \pi^{\vee, U} \to \Hom(J_{U}, J_{U}^{\vee})\ot {M}$$
 be defined by ${\rm T}_{\rm alg}(f_{1}, f_{2}):= f_{2}^{\vee}\circ f_{1}$.

\begin{prop} \lb{def Z}
 If $\phi_{v}=\phi_{r,v}$ is as in \eqref{phip} for all $v\vert p$, then for any sufficiently large $r'$, the geometric kernel  
 $${\wtil Z}(\phi^{\infty},\chi)$$
  is invariant under $\prod_{v\vert p}K_{1}^{1}(\vpi_{v}^{r'})$, and it satisfies
\begin{align}\lb{Zgd}
 \lf(\tZ(\phi^{\infty}, \chi))=2 |D_{F}|^{1/2} |D_{E}|^{1/2}    L(1, \eta)
 \cdot 
 \langle {\rm T}_{\rm alg, \iota_{\frakp}}(\theta_{\iota_{\frakp}}(\vphi,\alpha_{p} |\cd |_{p}(\vpi)^{-r'}w_{r'}^{-1}\phi)) P_{\chi}, P_{\chi}^{-1}\rangle_{X}.
 \end{align}
\end{prop}
\begin{proof}
 The invariance under  $\prod_{v\vert p}K_{1}^{1}(\vpi_{v}^{r'})$ follows from the invariance of $\phi_{r}$ under the action of $\smalltwomat 1{\OO_{F, p}}{}{1}$ and the continuity of the Weil representation.
The proof of \eqref{Zgd} is indicated in \cite[proof of Proposition 5.4.3]{1} (with the correction of Appendix \ref{app err}).
\end{proof}

\subsection{Kernel identity}  We state our kernel identity and recall how it implies the main theorem.

\paragraph{Assumptions on the data} 
Consider the following local assumptions on the data at primes above $ p$.

\begin{enonce}[remark]{Assumption} \lb{assp1} {Let $U_{F,v}^{\circ}:=1+\vpi_{v}^{n}\OO_{F,v}$ with $n\geq 1$ be such that $\omega_{v}$ is invariant under $U_{F, v}^{\circ}$.  The character $\chi_{p}$ is \emph{sufficiently ramified} in the sense that it is nontrivial on $$V_{p}^{\circ}:=\prod_{v\vert p}q_{v|\OO_{E,v}}^{-1}(U_{F,v}^{\circ})\subset \OO_{E,p}^{\ts}.$$
(Recall from \S~\ref{sec:not} that $q_{v}\colon E_{v}\to F_{v}$ is the norm map.)}
\end{enonce}

Under this assumption, we  have  $t_{\chi}=t_{\chi}^{0}$; see   [I, Proposition 8.1.1.3], where $\xi_{U}\in {\rm Cl}(X_{U})_{\Q}$ denotes the \emph{Hodge} class defining the section ${\rm Cl}(X_{U})_{\Q}\to {\rm Cl}^{0}(X_{U})_{\Q}$. The technical advantage gained, which is the same as in [I] and is implicitly reaped in Theorem \ref{theo local comp} below, is that one may analyse the height generating series purely in terms of pairs of  CM divisors of degree zero, thus  avoiding a study of $\xi_{U}$ and the recourse to $p$-adic Arakelov theory made in \cite{dd-ant}.

\begin{enonce}[remark]{Assumption} \lb{assp2} {For each $v \vert p$, the open compact $U_{v}\subset \B_{v}^{\ts}$ 
satisfies:
\begin{itemize}
\item
 $U_{v}= U_{v,r}=1+\vpi^{r}M_{2}(\OO_{F,{v}})$ for some $r\geq 1$;
\item the integer $r\geq n$ is sufficiently large so that the characters 
 $\chi_{v}$ and $\alpha_{v}\circ q_{v}$ of $E_{v}^{\ts}$  are invariant under $U_{v,r}\cap \OO_{E,v}^{\ts}$.
\end{itemize}}
\end{enonce}

\paragraph{Convention on citations from \cite{1}} In \cite{1}, we have denoted by   $S_{\text{nonsplit}}$ the set  of places of $F$ nonsplit in $E$, and by  $S_{p}$ the set  of places of $F$ above $p$. When referring to results from \cite{1}, we henceforth stipulate that one should read any assumption such as   `let $v\in S_{\text{nonsplit}}$'  or `let $v$ be a place in $F$ nonsplit in $E$' as `let $v\in S_{\text{nonsplit}} -S_{p}$'.  Similarly, the set $S_{1}$ fixed in \cite[\S 6.1]{1} should be understood to consist only of places not above $p$.

\begin{theo}[Kernel identity] \lb{ker id}
Assume the hypotheses  of Theorem \ref{MT}, and that   $U$, $\vphi^{p}$, $\phi^{p\infty}$, $\chi$, $r$ satisfy   the assumptions of \cite[\S 6.1]{1} as well as  Assumptions \ref{assp1}, \ref{assp2}.  Let $\phi_{p}:= \ot_{v\vert p}\phi_{v, r}$ with $\phi_{v,r}=$\eqref{phip}. Then 
$$\lf({\rm d}_{F}\calI(\phi^{p\infty}; \chi))= {2 |D_{F}|  L_{(p)}(1, \eta)}\cdot   \lf(\tZ(\phi^{\infty}, \chi)).$$
\end{theo}
The elements of the proof will be gathered in \S~\ref{35}. 
\begin{lemm}\lb{implic}  Theorem  \ref{ker id} implies Theorem \ref{MT2}.
\end{lemm}
\begin{proof}
 As in \cite[Proposition 5.4.3]{1} corrected in Appendix \ref{app err}, we consider the following equivalent  (by \cite[Lemma  5.3.1]{1}) form of  the identity of Theorem \ref{MT2}:
\begin{align}\label{315}
\langle {\rm T}_{\rm alg, \iota_{\frakp}}(f_{1}\otimes f_{2}) P_{\chi}, P_{{\chi}^{-1}}\rangle_{J}
={      \zeta_{F}^{\infty}(2)\over  (\pi^{2}/2)^{[F:\Q]}  |D_{E}|^{1/2}  L(1, \eta)} 
\prod_{v|p}Z_{v}^{\circ}( \alpha_{v}, \chi_{v})^{-1} 
\cdot{\rm d}_{F} L_{p, \alpha}(\sigma_{A,E})( \chi)\cdot Q(f_{1}, f_{2}, \chi)
\end{align}
where $\iota_{\frakp}\colon M\into L(\chi)$, 
and
$P_{\chi} =\dashint_{[T]} {\rm T}_{t}(P-\xi_{P})\chi(t)\, dt \in J(\baar{F})_{L(\chi)}$. By linearity, \eqref{315} extends to an identity that makes sense for any element  ${\bf f}\in \pi\ot \pi^{\vee}$.  By the multiplicity-one  result for $E_{\A^{\infty}}^{\ts}$-invariant linear functionals on each of  $\pi$, $\pi^{\vee}$, it suffices to prove \eqref{315} for one element ${\bf f}\in \pi\ot \pi^{\vee}$ such that $Q({\bf f} , \chi)\neq 0$ (cf. \cite[Lemma 3.23]{yzz}). 

We \emph{claim} that Theorem \ref{ker id} gives  \eqref{315} for ${\bf f}=\theta(\vphi, \phi)$, where:
\begin{itemize}
\item $\vphi_{\infty}$ is standard antiholomoprhic in the sense of \cite{1},  $\phi_{\infty}$ is standard in the sense of \cite{1}; 
\item for all $v\vert p$,  $\vphi_{v}=$\eqref{vphip} and $\phi_{v}=\phi_{v,r}=$\eqref{phip} for any sufficiently large $r$.
\end{itemize}
The claim follows  from  \eqref{def plf}, \eqref{Zgd}, the local comparison  between $\cR_{v}^{\natural} $ and $Q_{v}\circ\theta_{v}$ for $v\nmid p$ of  \cite[Lemma 5.1.1]{1},  and the local calculation at $v\vert p$ of  Proposition \ref{Qp}. 

Finally, the existence of $\vphi, \phi$ satisfying both the required  assumptions and   $Q(\theta(\vphi, \phi))\neq 0$ follows from \cite[Lemma 6.1.6]{1} away from $p$, and the explicit formula of Proposition \ref{Qp} at $p$. 
\end{proof}

\subsection{ Derivative of the analytic  kernel} 
 We start by studying  the  incoherent Eisenstein series $\mathscr{E}(\phi_{2}^{p\infty})$. For $a\in F_{v}^{\ts}$, denote by 
 $$W_{a , v}^{\circ}$$
  the normalised local Whittaker function of $E_{r}(\phi_{2}^{p\infty}(\chi_{F}){\phi_{2,p}^{\circ}}, \chi_{F})=\eqref{Er}$, defined as in \cite[Proposition 3.2.1 and paragraph following its proof]{1}.
 
  The following  reviews and corrects   \cite[Proposition 7.1.1]{1}. 

\begin{prop}\label{6.1-2-3}  For each $v\vert p$, let  $\phi_{2, v}=\phi_{2, r,v}$ be as in \eqref{phip}.
\begin{enumerate} 
\item\label{6.1} Let $v$ be a  place of $F$ and $a\in F_{v}^{\times}$.
	\begin{enumerate}
	\item If $a$ is not represented by $({\bf V}_{2,v}, uq)$ then $W_{a,v}^{\circ}(g,u, \one)=0$.
	\item\label{(b)} (Local Siegel--Weil formula.)\quad If $v\nmid p$ and  there exists  $x_{a}\in {\bf V}_{2,v}$ such that $uq(x_{a})=a$, then 
	$$W_{a,v}^{\circ}(\smallmat y{}{}1,u, \one)= \int_{E_{v}^{1}}r(\smallmat y{}{}1, h)\phi_{2,v}(x_{a},u)\, dh$$
	\item\label{(c)} (Local Siegel--Weil formula at $p$.) \quad  If $v\vert p$, $a\in -1+\vpi^{r}\OO_{F,v}$, $u\in 1+\vpi^{r}\OO_{F,v} $, let  $x_{a}\in {\bf V}_{2,v}$ be such that $uq(x_{a})=a$. Then 
\beq\lb{Wap}
W_{a,v}^{\circ}(\smallmat y{}{}1,u, \one)=  |d|_{v}\int_{E_{v}^{1}}r(\smallmat y{}{}1, h)\phi_{2,v}(x_{a},u)\, dh.\eeq
\end{enumerate}
\item\label{6.3} For any   $a\in F_{v}^{\ts}\cap \prod_{v\vert p }( -1+\vpi^{r}\OO_{F,v})$, $u\in F^{\ts}\cap\prod_{v\vert p} (1+\vpi^{r}\OO_{F,{v}})$, there is a place $v\nmid p$ of $F$ such that $a$ is not represented by $({\bf V}_{2}, uq)$.
\end{enumerate}
\end{prop}
\begin{proof} Parts \ref{6.1}(a)-(b) are  as in \cite[Proposition 7.1.1]{1}.  Before continuing, observe that, under our assumptions, $a$ is always represented by  $({\bf V}_{2,v}, uq_{v})$ for all $v\vert p$:  this is clear if $v$ splits in $E$ and, up to possibly enlarging the integer $r$, it  may be seen  by   the local constancy of $q_{v}$ and the  explicit identity $q_{v}({\frak j}_{v})=-1$, where ${\frak j}_{v}$=\eqref{jv} below if $v$ is nonsplit. Then part \ref{6.1}(c) follows by explicit computation of both sides (starting e.g. as in \cite[proof of Proposition 6.8]{yzz} for the left-hand side). Explicitly, 
we have
$$\eqref{Wap}=
\begin{cases} 
e_{v}^{-1}|d|_{v}\vol(E^{1}_{v}\cap \OO_{E,v}^{\ts}, dh)|y|^{1/2} \one_{\OO_{F,v}}(ay)\one_{\OO_{F,  v}^{\ts}}(y^{-1}u)
  & \textrm{\ if \ } v(a)\geq 0 \textrm{\ and \ } v(u)=0 \\  
 0 & \textrm{\ otherwise.} 
\end{cases}$$

Finally, part 2 follows from the observation of the previous paragraph and \cite[Lemma 6.3]{yzz}.
\end{proof}

\begin{lemm}\lb{l th} Suppose that, for all $v\vert p$, $\phi_{1,v}=\phi_{1,r,v}$ is as in \eqref{phip1}. Then, for any $t\in T(\A)$,
the $a^{\rm th} $ $q$-expansion coefficient of the the theta series  of \eqref{theta def} vanishes unless
$$a\in \bigcap_{v\vert p} 1+\vpi^{r}\OO_{F,v}.$$
\end{lemm}
\begin{proof} Straightforward. \end{proof}

Denote   $$\Up_{p, *}^{r}:= (\prod_{v\vert p}\Up_{v, *})^{r} , $$ 
an operator on modular forms that  extends to an operator on all the spaces of $p$-adic $q$-expansions defined so far.

\begin{coro}\label{6.2} Suppose that $\phi^{p\infty} \in {\calS}({\bf V}_{2}^{p\infty}\times {\A^{p\infty, \times}})$ satisfies the assumptions of \cite[\S 6.1]{1}. 
  Let $\chi\in \Y_{\omega}^{\rm l.c.}$. For any sufficiently large $r$, we have
$$\Up_{p, *}^{r} \mathscr{I}(\phi^{p\infty}; \chi)=0.$$
\end{coro}
\begin{proof}
For the first assertion, we need to  show that the ${a}^{\rm th}$ reduced $q$-expansion coefficient of $\cI$ vanishes for all $a$ satisfying $v(a)\geq r$ for all $v\vert p$. By the  defining property \ref{ItE}  of $\cI(\chi')$ and the choice of $\phi_{p}$, the group $T(F_{p})\subset T(\A)$ acts trivially on the  $q$-expansion coefficients of  $\cI(\chi')$. The remaining integration on $T(F)Z(\A)V^{p}\bks T(\A)/T(F_{p}) $ is a finite sum, so  the coefficients $a$ is a sum of products of  the coefficients of index $a_{1}$ of $\theta $ and of index $ a_{2}$ of  $\cE(\one)$, for pairs $(a_{1}, a_{2}) $ with  $a_{1}+a_{2}=a$. When $a=0$, the vanishing follows from the vanishing of the constant term of $\cE$, which is proved as in \cite[Proposition 6.7]{yzz}. For $a\neq 0$, by Lemma \ref{l th} only the pairs $(a_{1}, a_{2})$ with $a_{1}\in \bigcap_{v\vert p}1+\vpi^{r}\OO_{F,v}$ contribute. If $v(a)\geq r$ for $v\vert p$, this forces $a_{2}\in  \bigcap_{v\vert p}-1+\vpi_{v}^{r}\OO_{F,v}$. Then the coefficient of index $a_{2}$ of $\cE(\one)$ vanishes by Proposition \ref{6.1-2-3}.
 \end{proof}

We can now proceed as in \cite[\S 7.2]{1}, except for the insertion of the operator 
$\Up_{p, *}^{r}$. 
(This will be innocuous for the purposes of Theorem \ref{ker id}, since the  kernel of $\Up_{p, *}^{r}$  is contained  in the kernel of $\lf$.)
We obtain, under the assumptions of \cite[\S 6.1]{1}, a decomposition  of \eqref{eqderan}  
\beq\lb{dec an}
\Up_{p, *}^{r}\cI'(\phi^{p\infty};\chi)= \sum_{v\in S_{\text{nonsplit}}- S_{p}} \Up_{p, *}^{r}\cI'(\phi^{p\infty};\chi)(v)
\eeq
valid in the space ${\bf S}$ of $p$-adic $q$-expansions with coefficients in $\Gamma_{F}\hat{\ot} L(\chi)$. See \emph{loc. cit.} for the definition of $\cI'(\phi^{p\infty};\chi)(v)$.

\subsection{Decomposition of the geometric kernel   and comparison}\lb{35} 
 Suppose that Assumptions \ref{assp1}, \ref{assp2} as well as the assumptions of \cite[\S 6.1]{1} are satisfied.   Then we may decompose  (see \cite[\S 8.2]{1}) the generating series  \eqref{the series} as
\beq\lb{eqdec} 
{\wtil Z}(\phi^{\infty},\chi) =\sum_{v}{\wtil Z}(\phi^{\infty},\chi)(v), \eeq
according to the decomposition $\la \ , \ \ra_{X}=\sum_{v} \la \ , \ \ra_{X, v}$ of the height pairing. Here the sum runs  over all finite places of $F$.  

The following is the main result of \cite{1} on the local comparison away from $p$.  Let $\baar{\bf S}{}'= \baar{\bf S}{}'_{S_{1}}$ be the quotient of the space of  $p$-adic $q$-expansions recalled above \eqref{bS'}. 

\begin{theo}[{\cite[Theorem 8.3.2]{1}}]\label{theo local comp} Let $\phi^{\infty}=\phi^{p\infty}\phi_{p}$ with $\phi_{p}=\prod_{v\vert p} \phi_{v}$ as in \eqref{phip}.  Suppose that Assumptions \ref{assp1}, \ref{assp2} as well as the assumptions of \cite[\S 6.1]{1} are satisfied. 
Then  we have the following identities of reduced $q$-expansions in $\baar{\bf S}{}'$:
\begin{enumerate}
\item\label{part split b} If $v\in S_{\rm split} - S_{p}$, then 
$$\tZ(\phi^{\infty}, \chi)(v)=0.$$
\item\label{part ns b} If $v\in S_{\textup{nonsplit}}- S_{1}-S_{p}$, then 
$$\Up_{p, *}^{r}\calI'(\phi^{p\infty}; \chi)(v)=2 |D_{F}| L_{(p)}(1,\eta) \Up_{p, *}^{r}\tZ(\phi^{\infty}, \chi)(v).$$
\item\label{part s1 b} If $v\in S_{1}$, then
\begin{align*}
\Up_{p, *}^{r} \calI'(\phi^{p\infty}; \chi)(v), \qquad \Up_{p, *}^{r}  \tZ(\phi^{\infty}, \chi)(v)
\end{align*}
are theta series attached to a quaternion algebra  over $F$.
\item\label{part sp} The sum 
$$\Up_{p, *}^{r}\tZ(\phi^{\infty}, \chi)(p):=\sum_{v\in S_{p}} \Up_{p, *}^{r} \tZ(\phi^{\infty}, \chi)(v)$$
belongs to the isomorphic image $\baar{\bf S}\subset \baar{\bf S}{}'$ of the space of $p$-adic modular forms ${\bf S}$.
\end{enumerate}
\end{theo}

By this theorem and the decompositions of $\cI'$ and $\wtil{Z}$, the proof of the kernel identity of Theorem \ref{ker id} (hence of the main theorem) is now reduced to showing  the following proposition. (See \cite[\S 8.3, last paragraph]{1} for the details of the deduction.)
\begin{prop}\lb{kill lp}
 Retain the assumptions of Theorem \ref{theo local comp}, and further assume that $V_{\frakp}A $ is potentially crystalline at all $v\vert p$.  Then the $p$-adic modular form
 $$\Up_{p, *}^{r}\tZ(\phi^{\infty}, \chi)(p) \in \baar{\bf S}$$ 
 is annihilated by $\lf$. 
\end{prop}

Let $S$ be a finite set of non-archimedean places of $F$ such that, for all $v\notin S$, all the data are unramified, $U_{v}$ is maximal, and $\phi_{v}$ is standard. Let $K=K^{p}K_{p}$
 be the level of the modular form $\tZ(\phi^{\infty})$, and let 
 $$T_{\iota_{\frakp}}(\sigma^{\vee})\in \mathscr{H}^{S}({L})=\mathscr{H}^{S}({M})\otimes_{M, {\iota_{\frakp}}} {L}$$
  be any $\sigma^{\vee}$-idempotent in the Hecke algebra as in \cite[Proposition 2.4.4]{1}.
  By that result, in order to establish Proposition \ref{kill lp} it suffices to prove that \begin{gather}\label{bbb}
 \lf( \Up_{p, *}^{r}T_{\iota_{\frakp}}(\sigma^{\vee}) \tZ(\phi^{\infty}, \chi)(p))=0.\end{gather}
  As in \cite{1}, we will in fact prove the following, which    implies \eqref{bbb}.

 \begin{prop}\label{is crit}  Let $v\vert p$. Under the assumptions of Proposition  \ref{kill lp}, for all $v\vert p$, the element
  $$T_{\iota_{\frakp}}(\sigma^{\vee})\tZ(\phi^{\infty}, \chi)(v)\in \baar{\bf S}{}'$$
   is $v$-critical in the sense of \eqref{defcrit}.
 \end{prop}
 The proof will occupy the following section.

\section{Local heights at $p$}
\lb{sec4}
The goal of this section is to prove Proposition \ref{is crit}, whose assumptions we retain throughout except for an innocuous modification to the data at the places $v\vert p$. Namely, let $\wtil{\phi}_{v}$ be the Schwartz function denoted by $\phi_{r}$ in \eqref{phip}, and let $\wtil{U}_{v}\subset \B_{v}^{\ts}$ be the open compact subgroup denoted by $U_{v, r}$  in  Assumption \ref{assp2}. Then we let 
$$U_{v}:=\wtil{U}_{v}U_{F,v}^{\circ}, \qquad \phi_{v}:= \dashint_{U_{F, v}^{\circ}}  r(z, 1) \wtil{\phi}_{v}\, dz.$$
Since $\chi_{v|F_{v}^{\ts}}=\omega_{v}^{-1}$ is by construction  invariant under $U_{F, v}^{\circ}$, the geometric kernels $\tZ(\phi^{p\infty}\wtil{\phi}_{p}, \chi) $ and $ \tZ(\phi^{p\infty}\phi_{p}, \chi)$ are equal.  Therefore we may work on the curve $X_{U}$.

We fix a place $v\vert p$. If $v$ splits in $E$ then the desired result is proven in \cite[\S 9]{1} with a correction in Appendix \ref{app err}. Therefore we may  and do assume that $v$ is nonsplit.  We denote by $w$ the place of $E$ above~$v$.

\medskip

We refer the reader to \S~\ref{11} for a general sketch of our argument. It is developed here as follows. In \S~\ref{41}, we prove that after acting by a high power of $\Up_{v,*}$, the coefficients of the generating series are height pairings with  CM points of high $v$-conductor (\emph{norm relation}). After some general background in \S~\ref{42}, we use the norm relation to prove the decay property of arithmetic intersection multiplicities in \S~\ref{43}. Finally, in \S~\ref{dec lh} we use again the norm relation and some $p$-adic Hodge theory to prove the decay property of local heights.

\subsection{Norm relation for the generating series}\lb{41}
The goal of this subsection will be to show that, for $s$ large enough, each  $q$-expansion coefficient of  $\Up_{p,*}^{s}\wtil{Z}(\phi^{\infty}, \chi)$  is a height pairing of CM  divisors of which one is  supported on Galois orbits of CM points of `pseudo-conductor'  $s$ (as defined below).

We start by considering the $\Up_{v, *}$-action on  the generating series  $\wtil{Z}(\phi^{\infty})=\eqref{zphi}$. Recall that $d_{v}$ (\S~\ref{sec:not}) is a generator of the different ideal of $F_{v}$.

\begin{lemm} \lb{U on Z}
If $a\in \A^{\infty\ts}$ satisfies $v(a)\geq -v(d_{v})$, the $a^{\text th}$ reduced
 $q$-expansion coefficient of $\Up_{v,*}{\wtil Z}(\phi^{\infty})$ equals
 $${\wtil Z}_{a\vpi_{v}^{}}(\phi^{\infty}),$$
where  $ {\wtil Z}_{a'}(\phi^{\infty})$ is defined in \eqref{za}. 
 \end{lemm}

\begin{proof} After computing the Weil action of $\Up_{v, *}$ on $\phi$, this is a simple change of variables.
\end{proof}

We can factor 
$$ {\wtil Z}_{a}(\phi^{\infty})  =  {\wtil Z}_{a^{v}}(\phi^{v\infty}) Z_{a_{v}}(\phi_{v}), $$
as the composition of the  commuting correspondences  
\begin{align} \nonumber
\tZ_{a}^{v}(\phi^{v\infty}) &:=c_{U^{p}}\sum_{x^{v}\in U^{v}\bks \B^{v\infty\times}/U^{v}}\phi^{v\infty}(x^{v}, aq(x^{v})^{-1})Z(x^{v})_{{U}}, \\
\lb{Zloc} Z_{a_{v}}(\phi_{v}) &:= \sum_{x_{v}\in U_{v}\bks \B_{v}^{\ts}/U_{v} } \phi_{v}
(x_{v}, aq(x_{v})^{-1})Z(x_{v})_{{U}}. 
\end{align}

From here until after the proof of Lemma \ref{Xireps}, we work in a local situation and drop $v$ from the notation.
Let $\theta\in\OO_{E}$ be such that $\OO_{E}= \OO_{F}+\theta\OO_{F}$, and 
write $\rT={\Tr_{E/F}}(\theta_{})$, $\rN_{} =\Nm_{E_{}/F_{}}(\theta_{})$.
Fix the embedding $E_{}\to \B_{}=M_{2}(F_{})$ to be 
\beqq
 t=a+\theta b\mapsto \twomat {a+b\rT} {b{\rN}} {-b} a.
\eeqq
Let \beq\lb{jv}\rj:= \twomat 1 \rT {}{-1}.\eeq 
Then $\rj^{2}=1$, $q({\rj})=-1$,  and for all $t\in E$, $\rj t =t^{c}\rj$; and in the orthogonal decomposition ${\bf V}={\bf V}_{1}\oplus {\bf V}_{2}$, we have 
$$\OO_{{\bf V}_{2}}=  {\bf V}_{2}\cap M_{2}(\OO_{F})=\rj \OO_{E}.$$

Let $\Xi(\vpi^{r}) =1+\vpi^{r}\OO_{E}+ \rj (\OO_{E}\cap q^{-1}(1+\vpi^{r}\OO_{F}))$; then   $\phi$ 
is a fixed multiple of the characteristic function of $\Xi(\vpi^{r}) \ts (1+\vpi^{r}\OO_{F})\subset \B \ts F^{\ts}$. For $a\in F^{\ts}$, let $$\Xi(\vpi^{r})_{a} =\{x\in \Xi(\vpi^{r}) \  |\ q(x)\in a(1 +\vpi^{r}\OO_{F})\}.$$ Then 
  the local component \eqref{Zloc} of the  generating series equals
$$ Z_{a}(\phi)  =
\sum_{x \in U\bks \Xi(\vpi^{r})_{a} U_{F}^{\circ}/U }
Z(x)_{{U}},  
$$
up to a constant that is  independent of $a$.

\begin{lemm} \lb{Xireps}
Let $a\in F^{\ts}$ satisfy $v(a)=s\geq r$.  The natural map $\Xi(\vpi^{r})_{a} U_{F}^{\circ}/U \to U\bks \Xi(\vpi^{r})_{a} U_{F}^{\circ} /U$ is a bijection.   For either quotient set, a complete set of representatives is given by the elements 
$$x(b):= 1+\rj b$$
as $b$ ranges through a complete set of representatives for 
$$q^{-1}( 1-a(1+\vpi^{r}\OO_{F})) /(1+\vpi^{r+s}\OO_{E})     \subset (\OO_{E}/\vpi^{r+s}\OO_{E})^{\ts}.$$
\end{lemm}
\begin{proof}  It is equivalent to prove the same statement for the quotients of  $\Xi(\vpi^{r})_{a}$ by the group $\wtil{U}=1+\vpi^{r} M_{2}(\OO_{F})$. 
By acting on the right with elements of $\OO_{E}\cap U$, we can bring any element of $\Xi$ to one of the form $x(b)$. Write any $ \gamma \in \wtil{U}$ as $\gamma =1+\vpi^{r}u_{1} +{\rm j} \vpi^{r}u_{2}$ with $u_{1}$, $u_{2}\in \OO_{E}$. Then 
$$x(b)\gamma = (1+\rj b)(1+\vpi^{r}u_{1} +\rj \vpi^{ r} u_{2}) =  1+ \vpi^{r}(u_{1}+b^{c}u_{2}) + \rj (b+ \vpi^{r}(bu_{1}+u_{2})$$
 is another element of the form $x(b')$ if and only if  $u_{1}=-b^{c}u_{2}$. In this case
 $$b'= b+\vpi^{r}(1-q(b))u_{2} 
 \in b(1+\vpi^{r+s}\OO_{E}).$$
Thus the class of $b$ modulo $\vpi^{r+s}$ is the only  invariant of the quotient $\Xi/\wtil{U}$.  The $\wtil{U}$-action on the left similarly preserves this invariant.
\end{proof}

We go back to a global setting and notation, restoring the subscripts $v$ and $w$.
Denote by ${\rm rec}_{E_{w}}\colon E_{w}^{\ts}\to \calG_{E,w}^{\rm ab}$ the reciprocity map of class field theory. Recall that for $x\in \B^{\infty\ts}$, we have a point $[x]_{U}\in X_{U}$; for a subset $\Xi'\subset \B^{\infty\ts}$, we similarly denote $[\Xi']_{U}=\{[x]_{U}\, |\, x\in \Xi'\}$. 

\begin{lemm} \lb{norm rel} Fix $a\in \OO_{F,v}$ with $v(a)=s\geq r$.
\begin{enumerate}
\item Let $b\in ( \OO_{E,w}/\vpi_{v}^{r+s}\OO_{E,w})^{\ts}$, $t\in 1+\vpi_{v}^{r}\OO_{E,w}$. Then 
$${\rm rec}_{E_{w}}(t)[x(b)]_{U}= [x(b t^{c}/t)]_{U}.$$
\item We have 
$$[\Xi(\vpi_{v}^{r})_{a}U_{F, v}^{\circ}]_{U}= \bigsqcup_{\baar{b}} \ {\rm rec}_{E_{w}}((1+\vpi_{v}^{r}\OO_{E,w})/(1+\vpi_{v}^{r}\OO_{F,v})(1+\vpi_{v}^{r+s}\OO_{E,w})) [x(\baar{b})]_{U},$$
where the Galois action  is faithful, and  $\baar{b}$ ranges through a set of representatives for
\beq \lb{theset}
q_{v}^{-1}( 1-a(1+\vpi_{v}^{r}\OO_{F,v})) /(1+\vpi_{v}^{r+s}\OO_{E,w})\cdot (1+\vpi_{v}^{r+v(\tht-\tht^{c}) }\OO_{E,w}\cap q_{v}^{-1}(1)).\eeq		
The size of the  set \eqref{theset} is bounded uniformly in $a$.
\end{enumerate}
\end{lemm} 
\begin{proof}
For part 1, we have 
$${\rm rec}_{E_{w}}(t)[x(b)]_{U}= [t+t\rj b]_{U}= [t+\rj t^{c} b]_{U} = [1+\rj b t^{c}/t]_{U}.$$
Part 2 follows  Lemma \ref{Xireps} and  part 1, noting that the group $(1+\vpi_{v}^{r+v(\tht-\tht^{c}) }\OO_{E,{w}})\cap q_{v}^{-1}(1)$ is the image  of the map $t\mapsto t^{c}/t$ on $1+\vpi_{v}^{r}\OO_{E,w}$.  Finally, the map $({\rm projection}, q_{v})$ gives an injection from 
$$q_{v}^{-1}( 1-a(1+\vpi_{v}^{r}\OO_{F,v})) /(1+\vpi_{v}^{r+s}\OO_{E,w})\cdot (1+\vpi_{v}^{r+v(\tht-\tht^{c}) }\OO_{E,w}\cap q_{v}^{-1}(1)) $$
to
 $$
(\OO_{E,w}/\vpi_{v}^{r+v(\tht -\tht^{c})}\OO_{E,w})^{\ts} \ts (1-a(1+\vpi_{v}^{r}\OO_{F,v}))/q(1+\vpi_{v}^{r+s}\OO_{E,w}), $$
whose size is bounded uniformly in $a$ (more precisely, the second factor is isomorphic to $\OO_{F,v}/{\rm Tr}(\OO_{E,v})$ via the map $1-a(1+\vpi_{v}^{r}x)\mapsto x$).
\end{proof}

We denote by $\baar{w}$ an extension of the place $w$ to $E^{\rm ab}$. For $s\geq 0$, let 
$$H_{s}\subset E^{\rm ab}$$
 be the finite abelian extension of $E$ with norm group $U_{F}^{\circ} U_{T}^{v}(1+ \vpi_{v}^{r_{v}+s}\OO_{E,v})$, where $U_{T}^{v}= U^{v} \cap E_{\A^{\infty}}^{\ts}$. Let $H_{\infty}= \bigcup_{s\geq 0} H_{s}$. If $r_{v}$ is sufficiently large, for  all $s\geq 0$ the extension $H_{s}/H_{0}$ is totally ramified at $\baar{w}$ of degree $q_{F, v}^{s}$, and 
$$\Gal(H_{s}/H_{0})\cong \Gal(H_{s, \baar{w}}/H_{0, \baar{w}}) \cong (1+{\vpi}_{v}^{r_{v}}\OO_{E,v} )/(1+\vpi_{v}^{r_{v}}\OO_{F,v}) (1+\vpi_{v}^{r_{v}+s}\OO_{E, v}).$$
In particular,
$$[H_{s}:H_{0}]= [H_{s, \baar{w}} : H_{0, \baar{w}}] = q_{F, v}^{s}.$$

We will say that a CM point $z\in X_{H_{0}}$ has \emph{pseudo-conductor} $s\geq 0$ (at $\overline{w}$) if $H_{0, \baar{w}}(z)=H_{s, \baar{w}}$.

\begin{prop}\lb{pseudo} There exists an integer $d>0$ such that for all  $a\in \A^{\infty,x}$ with $v(a)=s\geq r_{v}$, there exists a degree-zero divisor $D_{a}\in d^{-1} \Div^{0}(X_{U , H_{s}})$, supported on CM points, such that 
$$\tZ_{a}(\phi^{\infty})[1]_{U} = {\rm Tr}_{H_{s}/H_{0}} (D_{a})$$ 
in ${\rm Div}^{0}(X_{U, H_{0}}).$
All prime divisor  components of $D_{a}$ are CM points of pseudo-conductor $s$.
\end{prop}
\begin{proof} This follows from Lemma \ref{norm rel}.2, by taking $D_{a}$ to be a fixed rational multiple (independent of $a$) of 
$$\tZ_{a}^{v}(\phi^{v\infty}) 
\sum_{\baar{b}\in \eqref{theset}}  [x(\baar{b})]_{U}.$$
The divisor is of degree zero by \cite[Proposition 8.1.1]{1}. Its prime components are not defined over proper subfields $H_{s'}\subset H_{s}$ because of the faithfulness statement of   Lemma \ref{norm rel}.2.
\end{proof}

\subsection{Intersection multiplicities on arithmetic surfaces}\lb{42}
Before continuing, we  gather some  definitions and a key  result.
\subsubsection{Ultrametricity of intersections on surfaces}
Let $\X$ be a $2$-dimensional   regular Noetherian scheme, finite flat over a field $\kappa$ or  a discrete valuation ring $\OO$ with residue field $\kappa$. 
We denote by $(\ \cd \ )_{\X}$ the usual $\Z$-bilinear intersection-multiplicity pairing of divisors intersecting properly on $\X$;   for effective divisors $D_{j}$ ($j=1, 2)$ with $\OO_{D_{j}}=\OO_{\X}/ \cI_{j}$, it is defined by
$$(D_{1}\cd D_{2})_{\X}= {\rm length}_{\kappa}\OO_{\X}/(\cI_{1}+\cI_{2}).$$
The subscript $\X$ will be omitted when clear from context.

We will need  the following result of
 Garc\'ia Barroso, Gonz\'alez P\'erez and Popescu-Pampu.
  \begin{prop} \lb{ultra} 
  Let $R$ be a noetherian  regular local ring of dimension $2$, which is a flat module over a field or a discrete valuation ring. Let  $\Delta $ be any irreducible curve in $\Spec R$. Then the function 
\beqq
d_{\Delta}(D_{1}, D_{2}):= \begin{cases}
(D_{1} \cd \Delta) (D_{2} \cd \Delta)/(D_{1}\cd D_{2}) & \text{if $D_{1}\neq D_{2}$}\\
0 & \text{if $D_{1}=D_{2}$}
\end{cases}
\eeqq
is an ultrametric distance on the space of irreducible curves in $\Spec R$ different from $\Delta$.
 \end{prop}
 \begin{proof} 
 For those rings $R$   that further satisfy the property of containing $\bC$, this is proved  in \cite{ultrametric}. 
 The proof only relies on (i) the existence of  embedded resolutions of divisors in the spectra of such rings, and (ii)  the negativity  of the intersection matrix of the exceptional divisor of a projective birational morphism between spectra of such rings. 
Both results still hold under our  weaker assumptions: see 
  \cite[Theorem 9.2.26]{qliu} for (i) and 
 \cite[Theorem 9.1.27, Remark 9.1.28]{qliu} for (ii).
 \end{proof}

\subsubsection{Arithmetic intersection multiplicities}\lb{flex}
Suppose now that   $\X$ is a $2$-dimensional regular Noetherian scheme,  proper flat over a discrete valuation ring  $\OO$ with residue field $\kappa$. A divisor on $\X$ is called \emph{horizontal} (respectively \emph{vertical}) if each of the irreducible components of the support $|D|$ is flat over $\OO$ (respectively  contained in the special fibre $\X_{\kappa}$).  We  extend $(\ \cd \ ):=(\ \cd \ )_{\X}$  to a bilinear form $(\ \bullet \ )$ on  pairs of divisors on $\X$ sharing no common horizontal irreducible component of the support by 
$$(\X_{\kappa}  \bullet V ):= 0$$ 
if $V$ is any vertical divisor. 

 Denote by $X$ the generic fibre of $\X$. If $D\in \Div^{0}(X)$ with Zariski closure $\baar{D}$ in $\X$, a  \emph{flat extension} of $D$ is a  divisor $\widehat{D}\in \Div(\X)_{\Q}$ such that $\widehat{D}-\baar{D}$ is vertical and 
$$(\widehat{D}\bullet V)=0$$
for any vertical divisor $V$ on $\X$. A flat extension of $D$  exists and it is unique up to   addition of rational linear combinations  of the connected components of  $\X_{\kappa}$.

  The  arithmetic intersection multiplicity  on divisors with disjoint supports  in ${\rm Div}^{0}(X)$ is then defined by
$$m_{X}(D_{1}, D_{2}):=(\baar{D}_{1}\bullet \widehat{D}_{2}) = (\widehat{D}_{1}\bullet \baar{D}_{2}) \in \Q.$$

\subsection{Decay of intersection multiplicities}\lb{43}
We continue using the notation introduced in  \S~\ref{41}. Let $m_{\baar{w}}:= m_{X_{H_{0, \baar{w}}}}$. 
Developing the approximation argument sketched in the introduction, we will   show that  for any degree-$0$ divisor $D$ on $X_{H_{0}}$, we have
$$m_{\baar{w}}(\wtil{Z}_{a\vpi^{s}}(\phi^{\infty}) [1]_{U}, D)= O(q_{F,v}^{s})$$
in $L$, uniformly in $a$.

Let $U_{0,v}:= \GL_{2}(\OO_{F, v})\subset \B_{v}^{\ts}$, let $X_{0}:= X_{U^{v}U_{0,v}}$, let $\X_{0}$ be the canonical model of $X_{0, F_{v}}$ over $\OO_{F, v}$, which is smooth  (see \cite{carayol}), and let $\X_{0}'$ be its base-change to   $\OO_{H_{0, \baar{w}}}$.
 Let $\X$ be the integral closure of  $\X_{0}'$ in $X_{H_{0, \baar{w}}}$, which is a regular model of  $X_{H_{0, \baar{w}}} $ over   $\OO_{H_{0, \baar{w}}}$,  and let ${\rm p}\colon \X\to \X_{0}'$ be the natural map. Thus we have a diagram 
\begin{equation*}
\xymatrix{
    X_{H_{0, \baar{w}}} \ar[d]  \ar[r]  & \X \ar[d]^{\rm p}\\
X_{0,H_{0, \baar{w}}} \ar[r]\ar[d]& \X_{0}'=\X_{0, \OO_{H, \baar{w}}}\ar[d]\\
X_{0,E_{v}} \ar[r]& \X_{0, \OO_{E,{w}}}
}
\end{equation*}
 of curves and regular integral models. 
 (The bottom row will be used in proving Lemma \ref{int D} below.)

\subsubsection{Some intersection multiplicities}
As a preliminary, we first compute the intersection multiplicities of   Zariski closures of  CM points with the special fibre of $\X$,  then bound  their intersections with horizontal divisors.

We denote by $\kappa$ the residue field of $H_{0, \baar{w}}$, and by $k$ the algebraic closure of $\kappa$.  For a scheme $\mathscr{C}$ and a point $y\in \mathscr{C}$, we denote   $\mathscr{C}_{y}:=\Spec \OO_{\mathscr{C}, y}$.

\begin{lemm} \lb{int sp}
Let $z_{s}\in X_{H_{0}}$ be a CM  point with pseudo-conductor $s$,  let $\baar{z}_{s}$ be its closure in $\X$, and let  $y\in \X_{\kappa}$ be its reduction modulo $\baar{w}$. 
Then
$$(\baar{z}_{s} \cd [\X_{y,\kappa}])_{\X_{y}}  =[\kappa(y):\kappa] q_{F,v}^{s}.$$
\end{lemm}
(Recall that, following \S~\ref{sec:not},  $\X_{y, \kappa}$ denotes the special fibre of $\X_{y}$.)
\begin{proof}
We will deduce this from Gross's theory of quasicanonical liftings \cite{gross}, which we recall. 
The situation is purely local and we drop all subscripts $v$, $w$,  and $\baar{w}$. For a finite  extension $K\supset E$ contained in $E^{\rm ab}$,  let $K^{\rm un}$ be  the maximal unramified extension of $K$ contained in $E^{\rm ab}$ (thus the  residue field of $K^{\rm un}$ is identified with $k$).

 By \cite[\S 7.4]{carayol}, for any supersingular point $y_{0} \in \X_{0,\OO_{E}^{\rm un}}$, the completed local ring of  $\X_{0, \OO_{E}^{\rm un}}$ at $ y_{0}$ is isomorphic to  $ \OO_{E^{\rm un}}\llb u\rrb$ and it is the deformation ring of formal modules studied by Gross. 
The  main result of \cite{gross} is  that  for any CM point $z_{0}\in X_{0,{E}^{\rm un}}$, there exist a unique integer $t$ (the \emph{conductor} of $z_{0}$) such that the following hold. First, the field $E^{\rm un }(z_{0})$ is  the abelian extension $E^{(t)} $ of $E^{\rm un}$ with norm group $(\OO_{F}+\vpi_{v}^{t}\OO_{E})^{\ts}/\OO_{F}^{\ts}$, which is totally ramified of some degree $d_{t}$.
 Second,  the inclusion of the Zariski closure $\baar{z}_{0}\into \X_{0, \OO_{E}^{\rm un}}$ gives rise to a map of complete local rings 
$$  \OO_{E^{\rm un}}\llb u\rrb \to \OO_{E^{(t)}} =\OO(\baar{z}_{0})$$
which sends $u$ to a uniformiser $\vpi^{(t) }$ of $E^{(t)}$.  It follows that if $\mu_{t}$ is the minimal polynomial of $\vpi^{(t)}$, 
\beq \lb{gross int}
(\baar{z_{0}} \cd  \X_{0, k})_{\X_{0, \OO_{E}^{\rm un}, y_{0}}}
= \dim_{k} \OO_{E^{\rm un}}\llb u\rrb /(\vpi_{E}, \mu_{t}(u)) =  \dim_{k}\OO_{E^{(t)}}/ \vpi_{E}= d_{t}.\eeq

Consider now the situation of the lemma. By the projection formula,  
$$(\baar{z}_{s} \cd [\X_{y, \kappa}] )_{\X_{y}} =( \baar{z}_{s} \cd  {\rm p}^{*}\X_{0,\kappa_{0}} )_{\X_{y}}=( {\rm p}_{*}\baar{z}_{s} \cd \X_{0,\kappa})_{\X'_{0,y'}} =
 [H_{0}(z_{s})\colon H_{0}(z_{0,s})] \cd 
(\baar{z}_{0,s} \cd \X_{0,\kappa })_{\X'_{0,y'}}  $$
where $y'={\rm p}(y)$ and $z_{0, s}= {\rm p}(z_{s})\in X_{0, H_{0}}$ is a CM point.   The last intersection multiplicity is $[\kappa(y):\kappa]$ times the multiplicities of the  base-changed divisors  to the ring of integers of $H_{0}^{\rm un}$, where $z_{0, s}$ remains an irreducible divisor since $H_{0}(z_{0, s})\subset H_{0}(z_{s})$ is totally ramified over $H_{0}$. We perform such base-change to $\OO_{H_{0}^{\rm un}}$ without altering the notation. Let $z$ be the image of $z_{0,s}\in X_{H_{0}^{\rm un}}$ in $X_{E^{\rm un}}$, and let $t$ be the conductor of $z$; so $E^{(t)}\subset H_{0}^{\rm un}(z_{0,s})$. The fibre    above $z\cong \Spec E^{(t)}$ in $X_{H_{0}^{\rm un}}$ is 
$$\Spec E^{(t)}\otimes_{E^{\rm un}} H_{0}^{\rm un}= \Spec H_{0}^{\rm un}(z_{0,s})^{\oplus c}$$
for $c=[E^{(t)}: E^{\rm un}] \cd [H_{0}^{\rm un}:E^{\rm un}]/ [H_{0}^{\rm }(z_{0,s}): H_{0}^{\rm }] $, and $z_{0,s}$ is one of the factors in the right hand side. By the projection formula applied to $ \X_{0}\ts \Spec \OO_{H_{0}^{\rm un }} \to \X_{0}\ts \Spec \OO_{E^{\rm un}}$ and \eqref{gross int}, we have  
\beqq \
 [H_{0}(z_{s})\colon H_{0}(z_{0,s})] \cd (\baar{z}_{0,s} \cd \X_{0,\kappa })_{\X'_{0,y'}} 
&=[\kappa(y):\kappa] c^{-1} \cd
  [H_{0}(z_{s})\colon H_{0}(z_{0,s})] \cd [H_{0}^{\rm un}:E^{\rm un}] \cd 
d_{t} \\&=[\kappa(y):\kappa] \cd
 [H_{0}(z_{s}): H_{0}]  \cd d_{t}^{-1} \cd d_{t}= [\kappa(y):\kappa] q_{F}^{s},
 \eeqq
 as desired.
\end{proof}

\begin{lemm} \lb{int D}
Let $\Delta$  be an irreducible horizontal divisor in $\X$. {The intersection multiplicities $$(\Delta\cd \baar{z}) $$ are bounded by an absolute constant as $z$ ranges among CM points of sufficiently large pseudo-conductor reducing to $y$.}
\end{lemm}
\begin{proof}
The intersection multiplicity   $(\Delta\cd \baar{z})$ is bounded by the degree of the natural map ${\rm q}\colon \X\to \X_{0, \OO_{E,{v}}}$ near $y$, times the intersection multiplicities of the pushforward divisors to $\X_{0, \OO_{E,v}}$. Similarly to   the proof of Lemma \ref{int sp}, we may estimate this intersection in  the base change of $\X_{0}$ to $\OO_{E}^{\rm un}$. 
 The base-change of the divisor ${\rm q}_{*}z_{}$ equals a sum of  CM points of  $\X_{0, \OO_{E^{\rm un}}}$; because the extensions $H_{s}/H_{0}$ are totally ramified, the number of points in this divisor  is bounded by an absolute constant, and the conductors of all those CM points go to infinity with the pseudo-conductor~$s$ of~$z$. Thus it  suffices to show that if $\Delta_{0}$ is a fixed horizontal divisor  on $\X_{0, E^{\rm un}}$, its intersection multiplicity with CM points of conductor~$t$ is bounded as $t\to \infty$. 

 Let $z\in \X_{0, \OO_{E}^{\rm un}}$ be a CM point of conductor $t$, and let $y_{0}\in \X_{0, k}$ be the image of the reduction of $t$.  Now write  the image of $\Delta_{0}$ in the  completion of $\X_{0, \OO_{E^{\rm un}}}$ at $y_{0}$ as  $$ \widehat{\Delta}_{0}=\Spec \OO_{E^{\rm un}}\llb u\rrb /(f)\subset \Spec \OO_{E^{\rm un}}\llb u\rrb,$$ with $f=\sum_{i=1}^{d} a_{i}u^{i}$ an integral non-constant monic polynomial. Let ${\wtil f}\in k[u]$ be the  reduction of $f$.
 Then 
$$(\Delta_{0}\cd \baar{z})= 
\dim_{k}\OO_{E^{\rm un}}\llb u\rrb / (f, \mu_{t})=\dim_{k} \OO_{E^{(t)}}/(f(\vpi^{(t)}))
=\dim_{k} \OO_{E^{(t)}}/((\vpi^{(t)})^{\deg (\tilde{f})})=\deg (\tilde{f})\leq d$$
if $t$ is sufficiently large, since the normalised valuations of $\vpi^{(t)}$ decrease to $0$ as $t\to \infty$. This completes the proof of the lemma.
\end{proof}

\subsubsection{Approximation by vertical divisors}
The following proposition contains the essential new ingredient of this work.  We denote by
$$ {\rm CM}(X_{H_{0}})_{\geq s}\supset {\rm CM}(X_{H_{0}})_{s} $$ respectively the set of  CM points of $X_{H_{0}}$ that have pseudo-conductor  at least, or equal to, a given integer~$s$. We denote by $\cV$ the set of irreducible components of  $\X_{\kappa}$ (henceforth: `vertical components'), and if $y\in \X_{\kappa}$ is a closed point we denote by $\cV_{y}\subset \cV$ the set of vertical components of $\X_{y, \kappa}$.
 We still use a bar to denote Zariski closure. 

\begin{prop} \lb{approx}  There exist an integer $s_{0}\geq 1$, depending only on $X_{H_{0}}$, and a function
$$(V, \rho)\colon {\rm CM}(X_{H_{0}})_{\geq s_{0}} \longrightarrow \cV\ts \Q^{\ts}$$
 satisfying the following property. 

For every divisor $D\in \Div(\X)_{L}$, there exists a constant $s_{D}\geq s_{0}$ depending only on the support of $D$, such that if $z\in X_{H_{0}}$ is a CM point of conductor $s\geq s_{D}$, then 
 $(\baar{z}\cd D) $ may be computed as follows.
Let $V=V(z)$, $\rho=\rho(z)$, and write  
$$D= c\X_{\kappa} +D'$$
with $c\in L$ and $D'\in \Div(\X)_{L}$ a divisor whose support does not contain $V$.
 Then 
\beq\lb{int for} (\baar{z}\cd D) = 
c[\kappa(y):\kappa]  q_{F,v}^{s}+ \rho (V\bullet D),
\eeq
where
 $y\in \X_{\kappa}$ denotes the reduction of $z$ modulo $\baar{w}$.
\end{prop}
\begin{rema}
The vertical  component $V(z)$ will be   characterised as the one maximising  the intersection multiplicity with $\baar{z}$.
We refer the reader to \cite[\S~2; see also Figure 1 in \S~1]{equi} for an equivalent and possibly more vivid geometric description\footnote{Note however that the substantial results of \cite{equi}  hold for the curve $X$ \emph{over  the field $F_{v}$}.} of the relation of $V$ to $z_{s}$:
 one can define pairwise disjoint open subsets (`geometric basins') of  the Berkovich analytification of $X$, labelled by the irreducible components of the special fibre; then $z_{s}$ belongs to the basin corresponding to $V$.
\end{rema}
\begin{proof}
We will omit all subscripts $v$, $w$,  $\baar{w}$ and use some of the notation introduced in the proof of Lemma \ref{int sp}.

Let $y\in \X_{\kappa}$ be a closed point, and write  $[\X_{y, \kappa}]= \sum_{V' \in \cV_{y}} e_{V'}V'$ as divisors.
By Lemma \ref{int sp}, the weighted sum  
\beq\lb{the q}
\sum_{V'} {e_{V'} }
(\baar{z}_{s}\cd V')=(\baar{z}_{s}\cd [\X_{y, \kappa}])=  [\kappa(y)\colon \kappa] q_{F}^{s}\eeq
is independent of the choice of a CM point $z_{s}\in X_{H_{0}}$ of pseudo-conductor $s$ reducing to $y$. As \eqref{the q} goes to infinity with~$s$ (and the coefficients $e_{V'}$ are independent of $s$), the quantity
 \beq\lb{eq min} \textstyle{
 \max_{V'\in \cV_{}}\,
 (\baar{z}_{s}\cd V')}
  \eeq 
 (more precisely, the minimum of those maxima as $z_{s}$ varies in ${\rm CM}(X_{H_{0}})_{s}$) goes to infinity with $s$.

Fix now a CM point $z_{s}\in X_{H_{0}}$ of pseudo-conductor $s$, let $y\in \X_{\kappa}$ be its  reduction, and let  $V\in \cV_{y}$
 be a vertical  component realising the maximum in \eqref{eq min}. 
Let
 $D\neq V$ be an irreducible divisor in $\X_{y}$.  Pick any irreducible horizontal divisor   $\Delta \neq D, \baar{z}_{s}$ in $\X_{y}$, and consider the ultrametric distance $d_{\Delta}$ 
 of Proposition \ref{ultra} for $ R=\OO_{\X, y}$. (Note that $\Delta$ may be drawn from a finite set independent of $z_{s}$ and $D$; in fact we may fix  any set $\underline{\Delta}$ of at least two irreducible horizontal divisors that are not Zariski closures of CM points, and for given $D$ pick any $\Delta\in \underline{\Delta}-\{D\}$.)

By  the choice of $V$ and Lemma \ref{int D},  
 if $s$ is sufficiently large (a condition depending on $D$),   we have 
 $$d_{\Delta}(\baar{z}_{s}, V) = {(\baar{z}_{s}\cd \Delta) (V\cd \Delta)\over (\baar{z}_{s}\cd V)} < d_{\Delta}( V, D),$$
  so that by 
 Proposition \ref{ultra},
  $$d_{\Delta}(\baar{z}_{s}, D) =   d_{\Delta}( V, D).$$ 
Unwinding the definitions, 
\beq\lb{the for}
(\baar{z}_{s}\cd D)= \rho  (V\cd D)\eeq
for $ \rho:= {(\baar{z}_{s}\cd \Delta)/  (V\cd \Delta) }$. 
Applied to a vertical component $D=V'\neq V$, the formula \eqref{the for} together with Lemma \ref{int D} shows the uniqueness of the maximising $V =:V(z_{s})$ for large $s$; it  is   clear that  $\rho=:\rho(z_{s})$ is then  uniquely determined as well.
Now the intersection formula \eqref{int for} follows by linearity from \eqref{the for} and Lemma \ref{int sp}.
\end{proof}

\begin{coro}\lb{cococo} If $D\in \Div^{0}(X_{H_{0}})_{L}$ is any degree-zero divisor,  then  for all sufficiently large $s$ and all $a$,
$$m_{\baar{w}}( \wtil{Z}_{a\vpi^{s}}(\phi^{\infty}) [1]_{U}, D)=O(q_{F,v}^{s})$$
in $L$, where the implied constant can be fixed independently of $a$ and $s$.
\end{coro}
\begin{proof} Let $\widehat{D}$ be a flat extension of $D$ to a divisor on $\X$ (with coefficients in $L$),
and abbreviate $Z_{a,s}:=  \wtil{Z}_{a\vpi^{s}}(\phi^{\infty}) [1]_{U}$. Then   by  Propositions \ref{pseudo} and \ref{approx}, 
$$m_{\baar{w}}( Z_{a,s}, D)= (\baar{Z}_{a,s}\cd \widehat{D}) = A q_{F,v}^{s}+ \sum_{i} \lm_{i}(V_{i} \bullet \widehat{D})$$
for some vertical components $V_{i}\subset \X $ and some  $A$, $\lm_{i}\in L$.
 By the definition of flat extension, $ (V_{i} \bullet \widehat{D})=0$ for all $i$. The constant $A$ is a linear combination  of the constants $c$ in  \eqref{int for} (which depend only on $\widehat{D}$), with coefficients whose denominators are bounded by those of $Z_{a,s}$; by Proposition \ref{pseudo}, the latter are bounded independently of $a$ and $s$.
 \end{proof}

\subsection{Decay of local heights}\lb{dec lh}
 Recall that we need to prove (Proposition \ref{is crit}) that 
 $$T_{\iota_{\frakp}}(\sigma^{\vee})\tZ(\phi^{\infty}, \chi)(v)$$ is a $v$-critical element of $ \baar{\bf S}{}'$.
 
  As in \cite[\S 9.2, proof of Proposition 9.2.1]{1}, this is reduced to the following.  For any $s\geq 0$, denote by $\la \ , \ \ra_{s,\baar{w}}$ the local height  pairing on $X_{H_{s, \baar{w}}}$, which is valued in   $H_{s, \baar{w}}^{\ts}\hat{\ot} L$; and let $$\la \ , \ \ra_{\baar{w}}=  [H_{s,\baar{w}}:F_{v}]^{-1}  \cd N_{H_{s,\baar{w}}/F_{v}} ( \la \ , \ \ra_{s,\baar{w}}),$$ which is valued in $F_{v}^{\ts}\hat{\ot} L\subset \Gamma_{F}\hat{\ot}L$ and is compatible with varying $s$ by \cite[(4.1.6)]{1}.
Then we will show that for all $\baar{w}\vert v$ and all $a\in \A^{S_{1}\infty, \ts}$ with $v(a)=r_{v}$, we have 
\beq  \lb{decaying}
 \langle \tZ_{a\vpi^{s}}(\phi^{\infty})[1]_{U}, {\rm T}_{\iota_{\frakp}}(\sigma^{\vee})_{U}^{\rm t} t_{\chi}\rangle_{\baar{w}} 
 &= O(q_{F,v}^{s}) \qquad& \text{in } {F}_{v}^{\ts}\hat{\ot}L(\chi)\\
  \langle \tZ_{a\vpi^{s}}(\phi^{\infty})[1]_{U}, {\rm T}_{\iota_{\frakp}}(\sigma^{\vee})_{U}^{\rm t} t_{\chi}\rangle_{0, \baar{w}} &= O(q_{F,v}^{s}) \qquad &\text{in } H_{0, \baar{w}}^{\ts} \hat{\ot}L(\chi)\\
\eeq
where the second statement implies the first one.  Until Lemma \ref{val cool} below, the argument follows the lines of previous works \cite{nekovar, ari, 1}.

\subsubsection{The norm relation and heights} Denote by  $N_{s}$   the norm from $H_{s , \baar{w}}$ to $H_{0, \baar{w}}$,  let $L':= L(\chi)$, and let $\frakp'\subset \OO_{L'}$  be the maximal ideal. By the norm relation of Proposition \ref{pseudo}, the aforementioned compatibility \cite[(4.1.6)]{1},  and the integrality result of \cite[Proposition 4.3.2]{1},\footnote{When comparing with the similar argument of  \cite[\S 9.2, proof of Proposition 9.2.1]{1}, our field $H_{s}$ should be assimilated to the $H'_{s}$ of \emph{loc. cit.}} 
\beq \lb{longg}
 \langle \tZ_{a\vpi^{s}}(\phi^{\infty}) [1]_{U}, {\rm T}_{\iota_{\frakp}}(\sigma^{\vee})_{U}^{\rm t} t_{\chi}\rangle_{0, \baar{w}} 
&= \langle {\rm Tr}_{H_{s}/H_{0}} (D_{a\vpi^{s}}), {\rm T}_{\iota_{\frakp}}(\sigma^{\vee})_{U}^{\rm t} t_{\chi}\rangle_{0, \baar{w}}   \\
&= 
N_{s}( \langle  D_{a\vpi^{s}}, {\rm T}_{\iota_{\frakp}}(\sigma^{\vee})_{U}^{\rm t} t_{\chi}\rangle_{s,\baar{w}})
\\
& \in \frakp'^{-( d_{00}+ d_{0}+d_{1,s}+d_{2,s})}   N_{s}(H_{s, \baar{w}}^{\ts}\hat{\ot} \OO_{L'})
\eeq
for some integers $d_{i, (s)}\geq 0$ that we now define and study.

\subsubsection{Boundedness of denominators} 
Let $$V':=\pi_{A^{\vee}}^{U}\ot_{M}  V_{\frakp}A^{\vee}(\chi^{-1}) \quad \subset\quad  V:= V_{p}J_{U} \ot_{\Q_{p}}L'$$
considered as  $\calG_{E_{w}}$-modules; let $V''$  be its direct complement  in the decomposition of $V$  in \cite[(9.2.4)]{1}, and let 
 $0\to V'^{+}\to V'\to V'^{-}\to 0$ be the ordinary filtration analogous to \eqref{splitting}. If $?\in\{ ', '', {}'^{+}\}$,  
let $T^{?}:= T_{p}J_{U}\ot_{\Z_{p}}\OO_{L'}\cap V^{?}$, and let $T'^{-}=T'/T'^{+}$. 
 Then the integers $d_{i, (s)}$ are defined as follows:
\begin{itemize}
\item
$d_{00}$ accounts for the denominators of the divisors, and it can be taken to be independent of $s$ by Proposition \ref{pseudo}; 
\item $d_{0}$ is such that $\frakp'^{d_{0}} T\subset T'\oplus T''$;
\item 
$d_{1,s}:={\rm length}_{\OO_{L'}}  H^{1}({H}_{s,\baar{w}}, T''^{*}(1))_{\rm tors};$
\item $d_{2, s}:=  {\rm length}_{\OO_{L'}}  H^{1}_{f}(H_{s, \baar{w}}, T')/{N}_{\infty} H^{1}_{f}(H_{s, \baar{w}}, T')$, where  $N_{\infty}$ denotes the universal norms (\cite[\S~6]{nekheights}, \cite[\S~4.3]{1}) with respect to the infinite abelian extension of $H_{s, \baar{w}}$ cut out by the closure in  $\calG_{H_{s, \baar{w}}}^{\rm ab}\supset H_{s, \baar{w}}^{\ts} $ of 
$$\Ker [ H_{s, \baar{w}}^{\ts} \stackrel{N_{H_{s, \baar{w}}/F_{v}}}{\longrightarrow} F_{v}^{\ts} \to \Gamma_{F}\hat{\ot} L].$$
\end{itemize}

\begin{prop} \lb{di bounded}  Suppose that $V_{\frakp}A $ is  potentially crystalline as a representation of $\calG_{F_{v}}$; then the sequences of integers $(d_{1,s})$ and $(d_{2, s})$ are bounded. 
\end{prop}

We will use the following vanishing result, in which $\baar{L}$ denotes an algebraic closure of $L'$.

\begin{lemm}\lb{d1s} Let $\Gamma_{\infty}:=\Gal(H_{\infty, \baar{w}}/E_{ {w}})\cong E^{\ts}\bks E^{\ts}_{\A^{\infty}}/ U^{v}U_{F,v}^{\circ}$. 
For all Hodge--Tate characters $\psi\colon \calG_{E_{w}}\to \baar{L}^{\ts}$ factoring through $\Gamma_{\infty}$, and for any 
$$V^{?}\in \{V,  V''^{*}(1),  V'^{+, *}(1), V'^{-}\},$$
we have 
$$H^{0}(E_{w}, V^{?}(\psi))= 0.$$
\end{lemm}
\begin{proof}
The proof is largely  similar to that of \cite[Lemma 9.2.4]{1}, to which we refer for the background on the $p$-adic Hodge theory of characters. 

We have $$ H^{0}(E_{w}, V^{?}) = {\bf D}_{\rm crys}(V^{?}(\psi))^{\vphi=1},$$
where $\vphi$ is the crystalline Frobenius, and it suffices to prove that $ {\bf D}_{\rm crys}(V^{?}(\psi))^{\vphi^{d}=1}=0$ for  $d=[E_{w}:E_{w, 0}]$ where $E_{w,0}$ the maximal unramified extension of $\Q_{p}$ contained in $E_{w}$.     As $V'$ has been assumed potentially crystalline, it is pure of weight $-1$,  hence so are all the subquotients of $V'$ and $V'^{*}(1)$. In particular, $\vphi^{d}$ acts with negative weights on  ${\bf D}_{\rm crys}(V^{?})$ for $V^{?}=V'^{+, *}(1), V'^{-}$; 
by \cite[Theorem 5.3]{mokrane}, this last assertion is also true of  $V^{?}=V\cong V^{*}(1)$ and its subquotients such as $V^{?}=V''^{*}(1)$.
 Therefore, it suffices to show that  $\vphi^{d}$ acts with weight $0$ on $ {\bf D}_{\rm crys}(\psi^{m})$ for $m$ such that $\psi^{m}$ is crystalline.  

Since $\psi$ is trivial on $U_{F}^{\circ}$, the Hodge--Tate weights $(n_{\tau})_{\tau\in \Hom(E_{w}, \baar{L})}$ satisfy $n_{\tau}+n_{\tau c}=0$ where $c$ is the complex conjugation of $E_{w}/F_{v}$. The action of $\vphi^{d}$ on  $ {\bf D}_{\rm crys}(\psi^{m})$ is by
\beq\lb{twof}
\psi\circ {\rm rec}_{E, w}(\vpi_{w})^{-m}\cd  \prod_{\tau\in \Hom(E_{w}, \baar{L})}\vpi_{w}^{mn_{\tau}},\eeq
where $\vpi_{w}\in E_{w}$ is any uniformiser. Choose $\vpi_{w}$ so that $\vpi_{w}^{e(E_{w}/F_{v})}=\vpi_{v}$ is a uniformiser in $F_{v}$. Then $\vpi_{w}^{c}=\pm \vpi_{w}$, so that the second factor in \eqref{twof} is $\pm 1$. On the other hand, the subgroup $F^{\ts}\bks F_{\A^{\infty}}^{\ts}/U_{F, v}^{\circ}(U^{v}\cap F_{\A^{\infty}}^{\ts})\subset \Gamma_{\infty}$ is finite, hence $\vpi_{v}$ and $\vpi_{w}$ have finite order in $\Gamma_{\infty}$. It follows that the first factor  in \eqref{twof}  is  a root of unity too, hence $\vphi^{d}$ acts with weight $0$ on $ {\bf D}_{\rm crys}(\psi^{m})$.
\end{proof}

\begin{proof}[Proof of Proposition \ref{di bounded}]
By the long exact sequence attached to 
$$0\to T''^{*}(1)\to V''^{*}(1)\to T''^{*}(1)\otimes L'/\OO_{L'}\to 0$$ and the vanishing of $H^{0}({H}_{s,\baar{w}}, V'^{*}(1))$ (which follows from Lemma \ref{d1s}), we have
\beqq
H^{1}({H}_{s,\baar{w}}, T''^{*}(1))_{\rm tors}
&\cong
H^{0}({H}_{s,\baar{w}}, T''^{*}(1) \ot_{\OO_{L'}} L'/\OO_{L'}) \\
&=  H^{0}\left(E_{w}, T''^{*}(1)\otimes_{\OO_{L}} \OO_{L}[\Gal({H}_{s,\baar{w}}/E_{w})]\ot_{\OO_{L}} {L'/\OO_{L'}} \right).
\eeqq
By \cite[Theorems 6.6, 6.9]{nekheights}  (or strictly speaking, a slightly  generalised form thereof which still holds true by the arguments in \cite[proof of Proposition 4.3.2]{1}) and the vanishing of   $H^{0}({H}_{s,\baar{w}}, V'^{+, *}(1)\oplus V'^{-})$ (which follows from Lemma \ref{d1s}), we have
\beqq
d_{2, s}&\leq
 {\rm length}_{\OO_{L'}}  H^{0} ({H}_{s,\baar{w}}, T'^{+,*}(1) \ot_{\OO_{L'}} L'/\OO_{L'})  
+ {\rm length}_{\OO_{L'}}  H^{0} ({H}_{s,\baar{w}}, T'^{-} \ot_{\OO_{L'}} L'/\OO_{L'}) \\
&={\rm length}_{\OO_{L'}}
 H^{0}\left(E_{w},(T'^{+,*}(1) \oplus  T'^{-}) \otimes_{\OO_{L}} \OO_{L}[\Gal({H}_{s,\baar{w}}/E_{w})]\ot_{\OO_{L}} {L'/\OO_{L'}} \right).
 \eeqq
Then the boundedness of $d_{1,s}$ and $d_{2,s}$ follows  as in \cite[proof of Proposition 8.10]{ari} from the vanishing of
$$H^{0}(H_{\infty, \baar{w}},  V^{?})
\quad \subset \bigoplus_{\psi\colon \Gamma_{\infty}\to \baar{L}^{\ts} \text{ Hodge--Tate}} H^{0}(E_{w}, V^{?}(\psi))
$$
for $V^{?}\in \{ V''^{*}(1),V'^{+,*}(1), V'\}$, which is a consequence of Lemma \ref{d1s}.
\end{proof}

\subsubsection{Completion of the proofs}
We are ready to reduce our  decay statement for local heights to the decay  statement for intersection multiplicities proved in \S~\ref{43}.

\begin{lemm} \lb{val cool}
For all $s'\leq s$, the restriction of the $\baar{w}$-adic valuation yields an isomorphism of $\OO_{L'}$-modules
$$ \baar{w}\colon N_{s}(H_{s, \baar{w}}^{\ts}\hat{\ot}\OO_{L'}) / q^{s'}\cd( H_{0, \baar{w}}^{\ts}\hat{\ot}\OO_{L'}) 
\to
  \OO_{L'}/q^{s'}\OO_{L'}.$$
\end{lemm}
\begin{proof}
We drop all subscripts $\baar{w}$. Recall that the extension $H_{s}/H_{0}$ is totally ramified of degree $q^{s}$. 
Let $\vpi_{s}\in \OO_{H_{s}}$ be a uniformiser; then  $\omega_{0}:=N_{s}(\vpi_{s})$ is a uniformiser of $H_{0} $. For $*=0,s$ we have the decompositions $$H_{*}^{\ts}\hat{\ot} \OO_{L'} =\OO_{H_{*}}^{\ts}\hat{\ot}\OO_{L' }\oplus \vpi_{*}\ot\OO_{L'}.$$ The map $N_{s}$ respects the decompositions and, by local class field theory, has image 
$$q^{s}\cd(\OO_{H_{0}}^{\ts}\hat{\ot}\OO_{ L'}) \oplus \vpi_{0}\ot\OO_{L'}.$$
The valuation map annihilates the first summand and sends the second one isomorphically to $\OO_{L'}$. The result follows.
\end{proof}

\begin{proof}[Proof of Proposition \ref{is crit}]
By the comparison of the valuation-component of local heights with arithmetic intersections in
\cite[Proposition 4.3.1]{1},  applied to the curve $X_{U,H_{0}}$,  the image of the left-hand side of   \eqref{longg} under $\baar{w}$ is
\beq 
\lb{pollo}
  m( \tZ_{a\vpi^{s}}(\phi^{\infty})_{U}[1], {\rm T}_{\iota_{\frakp}}(\sigma^{\vee})_{U}^{\rm t} t_{\chi}).
\eeq
By Corollary \ref{cococo} applied to $D= {\rm T}_{\iota_{\frakp}}(\sigma^{\vee})_{U}^{\rm t} t_{\chi}$, 
 the right-hand side of \eqref{pollo} is $O(q_{F,v}^{s})$.  By \eqref{longg}, Proposition \ref{di bounded} and Lemma \ref{val cool}, we deduce the  desired decay statement \eqref{decaying}.
\end{proof}

\paragraph{Summary} We have just completed the proof of Proposition \ref{is crit}. It implies Proposition \ref{kill lp}, which together with  Theorem \ref{theo local comp} implies the kernel identity of Theorem \ref{ker id}. By Lemma \ref{implic}, that  implies Theorem \ref{MT2}, which is in turn  an equivalent form of Theorem \ref{MT} by Lemma \ref{is eq}.

\bigskip

\paragraph*{\textsc{Acknowledgment}} I would like to thank  the referees for a  sharp reading.

\appendix
\section{Local integrals}
Throughout this appendix, $v$ denotes a place of $F$ above $p$ unless specified otherwise. We use some of the notation introduced in \S~\ref{sec:not}, in particular the Weil representation $r$ (see \cite[\S3.1]{1} or \cite{yzz} for the formulas defining it).

\subsection{Interpolation factors} 

We relate the  interpolation factors of the $p$-adic $L$-function of this paper with those from  \cite{1}.
\begin{lemm} \lb{int gamma}
Let $\xi\colon E_{w}^{\ts}\to \bC^{\ts}$ and $\psi \colon E_{w} \to \bC^{\ts}$ be  characters, with $\psi\neq \one$. Let $dt$ be a Haar measure on $E_{w}^{\ts}$.
Then 
$$\int_{E_{w}^{\ts}} \xi(t) \psi(t) dt={d^{}t \over d^{}_{\psi}t}  \cdot \xi(-1)\cdot  \gamma(\xi, \psi)^{-1} .$$
\end{lemm}
The left-hand side is to be understood in the sense of analytic continuation from characters $\xi|\cd |^{s}$ for $\Re(s)\gg0 $.
\begin{proof} We may fix $d^{}t= d^{}_{\psi}t$. Then the result  follows from the functional equation for $\GL_{1}$ (\cite[(23.4.4)]{bh}):
\beq \lb{feqgl1}
Z(\phi,  \xi) = \gamma(\xi, \psi)^{-1} Z(\hat{\phi}, \xi^{-1}|\ |) ,
\eeq
where for a Schwartz function $\phi$ on $E_{w}$, 
$$Z( \phi, \xi) := \int_{E_{w}^{\ts}} \phi(t) \xi(t) d_{\psi}^{}t , \qquad \hat{\phi}(t):= \int_{E_{w}}\phi(x) \psi(xt) d_{\psi}x.$$
Namely, we  insert in \eqref{feqgl1} the  function $$\hat{\phi}:=\delta_{-1+\vpi_{v}^{n}\OO_{F}}:= \vol(1+{\vpi}_{v}^{n}\OO_{F, v}, d_{\psi }t)^{-1}\one_{-1+{\vpi_{v}}^{n}\OO_{F, v}}, \qquad n\geq 1,$$  approximating   a delta function at $t=-1$. Then by Fourier inversion $\phi(t)=\hat{\hat{\phi}}(-t) = \delta_{1+\vpi_{v}^{n}\OO_{F}}*\psi(t) $, which  if $n$ is sufficiently large (depending  on the conductor of $\xi$)    has the same integral against $\xi$ as $\psi(t)$. 
\end{proof}

\begin{prop} \lb{propcratio}
The ratio  $C(\chi_{p}') $ defined in  \eqref{cratio} is a constant $C\in L$ independent of $\chi'_{p}$. 
\end{prop}
\begin{proof}  By the definition of $e_{p}(V_{(A,\chi')})$ and a comparison of \cite[Lemma A.1.1]{1} with Lemma \ref{int gamma}  applied to $\prod_{w\vert v}\chi'_{w}\alpha_{v}|\cd|_{v}\circ q_{w}$, we have 
$$ C(\chi_{p}')= \prod_{v\vert p }  \gamma (\ad(W_{v}(1)^{++}), \psi_{v})^{-1}\in L.$$
\end{proof}

\subsection{Toric period at $p$}
We compare the toric period at a $p$-adic place with the interpolation factor.  Denote by $P_{v} \subset \GL_{2}(F_{v})$ the upper triangular Borel subgroup.

\begin{lemm} \lb{iw dec} The quotient space $K_{1}^{1}(\vpi^{r'})\bks\GL_{2}({F_{v}}) /P_{v}$ admits the set of representatives 
$$n^{-}(c):=\begin{cases} \smalltwomat 1{}c1& \text{if } c\neq \infty\\
\smalltwomat {}1{-1} {} & \text{if } c=\infty,
\end{cases}\qquad c\in \OO_{F, v}/\vpi^{r'}\OO_{F,v}\cup \{\infty\}.$$
\end{lemm}
\begin{prop}\lb{Qp}  Let $\chi\in \Y^{\rm l.c.}$  be a finite-order character, let $r$ be sufficiently large (that is, satisfying the $v$-component  of  Assumption \ref{assp2}), and let $W_{v}$ be as in $\eqref{Wv}$, $\phi_{v}=\phi_{v,r}$ be as in \eqref{phip}.

Let $\pi_{v}=\sg_{v}$ and  let $(,)_{v}\colon \pi_{v}\ts \pi_{v}^{\vee}\to L$ be a duality pairing satisfying the compatibility of \cite[(5.1.2)]{1} with the local Shimizu lift.\footnote{In \emph{loc. cit.}, the pairing $(,)_{v}$ is denoted by $\mathscr{F}_{v}$.}  Finally, let $Z^{\circ}_{v}(\alpha_{v}, \chi_{v})$ be the interpolation factor of the $p$-adic $L$-function of \cite[Theorem A]{1}.

Then  for all sufficiently large $r'>r$, 
$$ Q_{(, ), v, d^{\circ}t_{v}}(\theta_{v}(W_{v}, \alpha|\cd|_{v}(\vpi_{v})^{-r'_{v}} w_{r',v}^{-1}\phi_{v}),\chi_{v}) = |d|^{2}_{v}|D|_{v}\cd L(1, \eta_{v})^{-1}\cdot Z^{\circ}_{v}(\alpha_{v}, \chi_{v}), $$
where $Q_{(, ), v, d^{\circ}t_{v}}$ uses the measure $d^{\circ}t_{v}= |d|^{-1/2}_{v}|D|_{v}^{-1/2}dt$.
\end{prop}

\begin{proof}
We drop all subscripts $v$, and assume as usual that  $\psi$ is our fixed character of level $d^{-1}$. Let \begin{align}\label{Qsharp}
 Q_{(, )}^{\sharp}(f_{1}, f_{2}, \chi_{v})=
\int_{E^{\times}/F^{\times}} \chi(t) (\pi(t)f_{1}, f_{2})\,  {dt},
\end{align}
where $dt$ is the usual  Haar measure on $E^{\ts}/F^{\ts}$, giving volume $|d|^{1/2}|D|^{1/2}$ to $\OO_{E}^{\ts}/\OO_{F}^{\ts}$.

By the definitions  and  \cite[Lemma  A.1.1]{1} (which expresses  $Z^{\circ}$ as a normalised integral),   it  suffices to show that  
$$Q^{\sharp}(\theta_{}(W_{}, \alpha|\cd|(\vpi)^{-r'_{}}_{} w_{r'}^{-1}\phi_{}),\chi_{}) = 
|d|^{3/2}|D| \cd L(1, \eta_{})^{-1}\cdot \int_{E_{}^{\ts}} \alpha|\cd |\circ q(t) \chi(t)\psi_{E}(t) \, dt. $$

By  \cite[Lemma 5.1.1]{1} (which spells out a consequence of the normalisation of the local Shimizu lifting) and Lemma \ref{iw dec}
we can write 
$$Q^{\sharp}:= Q_{}(\theta_{}(W_{}, \alpha_{}|\cd|_{}(\vpi_{})^{-r'} w_{r'}^{-1}\phi_{r}),\chi_{v})= \sum_{c\in{\bf P}^{1}( \OO_{F}/\vpi^{r'} )} Q^{\sharp \, (c)} $$
where for each $c$, 
\begin{multline*} 
Q^{\sharp\,  (c)}_{}:=|d|^{-3/2}\cd\alpha|\ |(\vpi)^{-r}\int_{F^{\times}} W(\smallmat y{}{}1  n^{-}(c))\\
 \int_{T(F)}\chi(t) \int_{P(\vpi^{r'})\bks K_{1}^{1}(\vpi^{r'})} |y| r( n^{-}(c)   kw_{r}^{-1})\phi(yt^{-1}, y^{-1}q(t))
dk\, d^{}t \, {d^{\times }y\over |y|}.
\end{multline*}
Here $P(\vpi^{r'}) =P \cap K_{1}^{1}(\vpi^{r'})$.

It is easy to see that $Q^{\sharp (\infty)}=0$ (observe that $\phi_{2,r}(0)=0$). For $c\neq \infty$, we have $$n^{-}(c) w_{r'}^{-1 } = w_{r'}^{-1} \twomat 1 {-c\vpi^{-r'}} {}1, $$
and when $x=(x_{1}, x_{2})$ with $x_{2}=0$:
\begin{align*}
r(n^{-}(c )w_{r'}^{-1})\phi(x,u)&=
 \int_{\bf V}
  \psi_{E} (u x_{1}\xi_{1})      \psi(-uc q(\xi)) \phi_{r}(\xi, \vpi^{r'}u)  d\, \xi 
\end{align*}
On the  support of the integrand we have $v(u)=\vpi^{-r'}$ and $v(q(\xi))\geq r$, by the definition of $\phi_{r}$. If $v(c)< r'-r-v(d)$, the integration in $d\xi_{2}$ gives $0$;
hence  $Q^{\sharp (c)}=0$ in that case.

Suppose from now on that  $v(c)\geq r'-r-v(d)$. Then $\psi(-ucq(\xi))=1$ and 
$$r(n^{-}(c )w_{r'}^{-1})\phi(x,u)
  =  |d|^{3}|D|^{} L(1, \eta)^{-1} |\vpi|^{r}\psi_{E, r}(\vpi^{-r'}x_{1}) \delta_{1,U_{F},r}(\vpi^{r}u).$$
where $\psi_{E,r}:=\vol(\OO_{E})^{-1}\cd\delta_{1,U_{T,r}}*\psi_{E}$, and we have noted  that $\hat\phi_{2, r}(0)= e^{-1}|d| \cd\vol(q^{-1}(-1+ \vpi^{r}\OO_{F})\cap \OO_{{\bf V}_{2}})=|\vpi|^{r} |d|^{3}|D|^{1/2} L(1, \eta)^{-1}$.

If $r'$ is sufficiently large, the Whittaker function $W$ is invariant under  $n^{-}(c)$. Then 
\begin{multline*}
Q^{\sharp\, (c)}= 
 |d|^{3/2}|D| L(1, \eta)^{-1}\cd  |\vpi|^{r-r'}  \alpha(\vpi)^{-r'} \cd  |d| ^{1/2}\zeta_{F,v}(1)^{-1}
\\
\cdot
\int_{F^{\times}} W(\smallmat y{}{}1) 
 \int_{T(F)}\chi( t) 
 \psi_{E, r}(\vpi^{-r'}x_{1}) \delta_{1,U_{F,r}}(\vpi^{r'}y^{-1}q(t))
dt
  d^{\times}y,
\end{multline*}
where $|\vpi|^{-r'}|d|^{1/2}\zeta_{F,v}(1)^{-1}$ appears as $  \vol(P(\vpi^{r'})\bks K_{1}^{1}(\vpi^{r'}))$.

Integrating in $d^{\ts}y$ and summing the above over the $q_{F}^{r+v(d)}=|d|^{-1}|\vpi|^{-r}$ contributing values of $c$, we find 
$$Q^{\sharp}= |d|^{3/2}|D| \cd L(1, \eta)^{-1}\cd \int_{E^{\ts}}\chi(t)\alpha|\ |\circ q(t) \psi_{E}(t)\, dt,$$
as desired.
\end{proof}

\begin{coro}
\lb{non exc B}
If $\chi_{p}$ is not exceptional (that is, $e_{p}(V_{(A, \chi)})\neq 0$), then  the  quaternion algebra ${\bf B}$ over $\A$ satisfying $ \H(\pi_{ \B},\chi)\neq 0$ is indefinite at all primes $v\vert p$. 
\end{coro}
\begin{proof} By Propositions \ref{Qp} and \ref{propcratio}, if $\chi_{p}$ is not exceptional, then for all $v\vert p$ the  functional $Q_{v}\in \H(\pi_{M_{2}(F_{v})}, \chi_{v})\ot  \H(\pi_{M_{2}(F_{v})}^{\vee}, \chi_{v}^{-1})$ is not identically zero. 
\end{proof}

\section{Errata to \cite{1}} \lb{app err}
The salient mistakes are the following: the statement of the main theorem is off by a factor of $2$; the proof given needs a further assumption, (no stronger than) that $V_{\frakp}A$ is potentially crystalline at all $v\vert p$ (however the theorem still holds true without the assumption, cf. Remark \ref{rmk remove}); and the  Schwartz function $\phi_{2, p} $  given by the local  Siegel--Weil formula at $p$ needs to be \emph{different} from the Schwartz function used to construct the Eisenstein family. 

References  in \emph{italics}  are directed  to \cite{1}, references in straight letters to the present paper.
\begin{itemize}
\item \emph{Theorem A}.  
It should be 
$L_{p, \alpha}(\sigma_{E})\in \OO(\Y')^{\rm b}$
 (with the interpolation property  being correct for the choice of additive character $\psi_{p}$ as  in Theorem \ref{A}). For a correct discussion of the ring of rationality of $L_{p, \alpha}(\sigma_{E})$, within the context of a  generalised construction, see \cite[Corollary 4.5.4]{dd-pLf}.
\item \emph{Theorem B}. The constant factor should be $c_{E}$ and not $c_{E}/2$ (the latter is, according to \emph{(1.1.3)}, the constant factor of the Gross--Zagier formula in archimedean coefficients).\footnote{The heuristic reason for the difference is that the direct analogue of $s\mapsto L(1/2+s, \sigma_{E}, \chi)$ is $\chi_{F}\mapsto L_{p}(\sg_{E})(\chi\cd \chi_{F}\circ N_{E/F})$, whose derivative at $\chi_{F}=\one$ is twice our ${\rm d}_{F}L_{p}(\sg_{E})(\chi)$, as  the tangent map to  $\chi_{F}\mapsto \chi'=\chi\cd \chi_{F}\circ N_{E/F}\mapsto \omega^{-1}\cd \chi'_{|\A^{\ts}}$ is multiplication by~$2$.}
The mistake is introduced in the proof of \emph{Proposition 5.4.3}, see below.

The proof works under the further  assumption that $V_{\frakp}A$ is potentially crystalline at all $v\vert p$, see the correction to \emph{Proposition 9.2.1}.
\item \emph{Theorem C}. Similarly, the constant factor should be $c_{E}/2$ in part 3,  and $c_{E}$ in part 4.
\item \emph{\S2.1}. The space of $p$-adic modular forms is the closure of $M_{2}(K^{p}K_{1}^{1}(p^{\infty})_{p})$, not  $M_{2}(K^{p}K^{1}(p^{\infty})_{p})$.
\item \emph{Proposition 2.4.4.1}. The multiplier in \emph{equation (2.4.3)}  should be $\alpha|\cd |(\vpi^{-r})=\prod_{v\vert p } \alpha_{v}|\cd |_{v}(\vpi_{v}^{-r})$, and the statement holds for forms in $M_{2}(K^{p}K_{1}^{1}(p^{r})_{p})$.  Similarly, the definition of $R_{r,v}^{\circ}$ in \emph{Proposition 3.5.1} should have an extra $|\vpi_{v}|^{-r}$.  The result of \emph{Proposition A.2.2}, as modified below, holds true for this definition of $R_{r, v}^{\circ}$ (in \emph{loc. cit.}, a complementary mistake appears between the third-last and second-last displayed equations in the proof).
\item \emph{Lemma 3.2.2}. A factor  $\eta(y)$ is missing in the right-hand side of the formula.
\item  \emph{Proposition 3.2.3.2} should be  corrected as follows: \emph{Let $v\vert p$ and let $\phi_{2,v}=\phi_{2,v}^{\circ}$  be as in \eqref{phip} \emph{(of the present paper)}. Then 
$$ W_{a,r, v}^{\circ}(1, u, \chi_{F})=
\begin{cases} |d_{v}|^{3/2}|D_{v}|^{1/2}  \chi_{F,v}(-1)  & \textrm{\ if \ } v(a)\geq 0 \textrm{\ and \ } v(u)=0 \\  
 0 & \textrm{\ otherwise.} 
\end{cases}$$}
\item \emph{Equation (3.7.1)}: the right-hand side should have an extra factor of  $\prod_{v\vert p}|d|^{2}_{v}|D|_{v}$ owing to the correction to \emph{Proposition A.2.2}.
\item  \emph{Lemma 5.3.1}: the left-hand side of the last equation in the statement should be $\la f_{1}'(P_{1}), f_{2}'(P_{2})\ra_{J, *}$. 
\item \emph{{Proof of Proposition 5.4.3}}. The second-last displayed equation should have the factor of $2$ on the right hand side, not the left hand side:
$$2 \lf(\tZ(\phi^{\infty}, \chi))= |D_{F}|^{1/2} |D_{E}|^{1/2}    L(1, \eta)\langle {\rm T}_{\rm alg, \iota_{\frakp}}(\theta_{\iota_{\frakp}}(\vphi,\alpha(\vpi)^{-r}w_{r}^{-1}\phi)) P_{\chi}, P_{\chi}^{-1}\rangle.$$
Then the argument shows that, first, 
$$\langle {\rm T}_{\rm alg, \iota_{\frakp}}(f_{1}\otimes f_{2}) P_{\chi}, P_{{\chi}^{-1}}\rangle_{J}
={      \zeta_{F}^{\infty}(2)\over  (\pi^{2}/2)^{[F:\Q]}  |D_{E}|^{1/2}  L(1, \eta)} 
\prod_{v|p}Z_{v}^{\circ}( \alpha_{v}, \chi_{v})^{-1} 
\cdot{\rm d}_{F} L_{p, \alpha}(\sigma_{A,E})( \chi)\cdot Q(f_{1}, f_{2}, \chi)$$
(without an incorrect factor of $2$ introduced in the denominator of the right-hand side of (5.4.1) there);  second, that the above equation is equivalent to \emph{Theorem B} as corrected  above. 

An extra factor  $\prod_{v\vert p} |d|^{2}_{v}|D|_{v}$ should be inserted in the right-hand sides of the last and fourth-last displayed equation, cf. the corrections to \emph{(3.7.1)}, \emph{Proposition A.3.1}.
\item \emph{Proposition 7.1.1(b)}: should be replaced by Proposition \ref{6.1-2-3}(b-c).
\item  \emph{\S7.2, third paragraph}: the coefficient in the second displayed equation should have $|D_{E}|^{1/2}$, not $|D_{E/F}|^{1/2}$, in the denominator.
\item  \emph{Lemma 9.1.1}  is corrected by Lemma \ref{U on Z}  (this does not significantly affect the rest).
\item  \emph{Lemma 9.1.5}: the extension $H_{\infty}$ is \emph{contained} in a relative Lubin--Tate extension. This is the only property used. 
\item \emph{Proposition 9.2.1}. The assumption that $V_{\frakp}A$ is potentially crystalline at all $v\vert p$ should be added. The bounded dependence on $s$ of the integer $d_{2}=d_{2,s}$ was not addressed; it holds true by the proofs of Proposition 4.4.1 and  Lemma 4.4.2, which work verbatim in the split case (under the comparison given in footnote 8 of \S~\ref{dec lh}).
 The definition of $d_{1,s}$ contains  an extra $\ot_{\OO_{L}} L/\OO_{L}$. 
\item{\emph{Lemma 9.2.4}}: the statement should be that $H^{0}(\wtil{H}'_{\infty,\baar{w}}, V_{p}J_{U}^{*}(1))$ vanishes, and it  this group that should appear in the left-hand side of the first displayed equation in the proof.
\item  \emph{Lemma A.2.1}: the list of  representatives is missing the element $n^{-}(\infty)=\smalltwomat 01{-1}0$, cf. Lemma \ref{iw dec}. 
\item \emph{Proposition A.2.2}:  the statement should be 
$$R_{r,v}^{\natural}(W_{v}, \phi_{v}, \chi_{v}'{}^{\iota}) = |d|^{2}_{v}|D|_{v}\cd Z_{v}^{\circ}( \chi'_{v}):=    |d|^{2}_{v}|D|_{v}\cd
 {\zeta_{F,v}(2)L(1, \eta_{v})^{2}   \over L(1/2, \sigma_{E,v}^{\iota}, \chi'_{v}{})} 
  {\prod_{w\vert v}Z_{w}(  \chi_{w}' )}.$$ 
  The factor $|d|^{2}_{v}|D|_{v}$ missing from \cite{1} should first appear in the right-hand side of the displayed formula for $r(w_{r}^{-1})\phi(x, u)$ in the middle of the proof.
\item \emph{Proposition A.3.1} should be replaced by Proposition \ref{Qp}.
\end{itemize}

 \section{Correction}

The approximation argument used to prove the decay of intersection multiplicities is flawed. In this correction, we give  an alternative argument in a similar spirit, based on an explicit form of the approximation that we deduce from \cite{equi}. This argument requires some bounds on the ramification, so that the main theorem is  weakened.  

Referring to the paragraph \emph{The nonsplit case} in \S~1.1, our general approximation result involved \emph{local} arithmetic intersections, and so it does not imply the vanishing of \emph{global} intersections with flat divisors in a proper local integral model. Instead, we revisit an idea of Perrin-Riou and apply an operator ``$U_{p}-1$''. Since this acts as a difference operator on the Fourier coefficients of our generating series,  we obtain the vanishing (up to multiples of $p^{s}$) once we prove, by inspection, that the relevant sequences of approximating vertical components are constant in the index~$s$.

I would like to thank Wei Zhang for pointing out the mistake.

\subsection{Corrected statement}  \lb{stat c}
We denote by $S_{p,{\rm ns}}$ the set of places of $F$ above~$p$ that are nonsplit in $E$. In Theorem~B,  the assumption that \emph{$\chi_{p}$ is sufficiently ramified} should be replaced by the assumption: $$\text{\emph{for each $v\in S_{p, {\rm ns}}$, $v$~is inert and  $\chi_{v}$ is unramified}}.$$

\begin{rema} \lb{rmk anti}
It should be possible to prove the theorem also (at least)  in the case where at some nonsplit places $v\vert p$, the representation $\pi_{v}$ is unramified and $\chi_{v}$ is arbitrary. While in principle not more difficult than the case treated here, this case would require introducing a larger number of changes in the setup, making for a cumbersome text. We thus prefer to defer  it  to a future work under a  different global approach.
\end{rema}

\subsection{The mistake} It occurs in Proposition 4.3.3, whose proof (with notation as in \emph{loc. cit.}) correctly shows that 
\beq\lb{int for C} (\baar{z}\cd D) = 
c[\kappa(y):\kappa]  q_{F,v}^{s}+ \rho (V\cdot D')_{y}.\eeq
The term $(V\cdot D')_{y}$ is a local intersection multiplicity at $y$, and it is not necessarily equal to the global intersection $(V\bullet D)$ on $\X$. Therefore, corresponding  terms in the formula displayed in the proof of Corollary \ref{cococo} do not necessarily vanish, as the definition of flat extensions invoked in that proof only applies to global intersection pairings.

\subsection{Correction}
We explain the strategy to prove the statement under the hypotheses of \S~\ref{stat c}.

\subsubsection{Setup} 
We discard Assumption 3.4.1 on $\chi_{p}$; as in the corrected statement, we assume instead that $v$ is inert and $\chi_{v}$ is unramified for all  $v\in S_{p, {\rm ns}}$. 
We suppose that $(\phi, U)$ satisfy the assumptions of [I, \S~ 6.1] as well as Assumption 3.4.2, and the following extra assumption. Let $T_{\iota_{\frakp}}(\sigma^{\vee})$ be a spherical $\sigma^{\vee}$-idempotent as in [I, Proposition 2.4.4], which we may take to be of degree zero; by [Ram], we may and do assume that $T_{\iota_{\frakp}}(\sigma^{\vee})$ is supported at \emph{split} places of $F$ where all the data is unramified.  
\begin{enonce}[remark]{Assumption} \lb{ass-idp} 
We have 
\beq\phi=T_{\iota_{\frakp}}(\sigma^{\vee}) \phi^{\flat}\eeq
for some $\phi^{\flat}$ satisfying the assumptions of [I, \S~ 6.1].
\end{enonce}
This assumption will have the same effect as Assumption 3.4.1, namely it ensures that the geometric kernel can be written in terms of height pairings of degree-zero divisors. 

Denote by ${\rm T}(\sigma^{\vee}) $ the Hecke correspondence on $X_{U}$ attached to $T_{\iota_{\frakp}}(\sigma^{\vee}) $ via [I, Lemma 5.2.2]; it has degree zero.  Then by the definitions and [I, Lemma 5.2.2]
 $${}^{\qqq}\tZ(\phi^{\infty}, \chi)_{U}=\langle {}^{\qqq}\tZ_{*}(\phi^{\flat, \infty})1, {\rm T}(\sigma^{\vee}) t_{\chi}\rangle,$$
 and in $\baar{\bf S}{}'$ we have the decomposition
$${}^{\qqq} \tZ( \phi^{\infty}, \chi) =\sum_{v}  \tZ( \phi^{\infty}, \chi)(v)$$
where 
\beq \lb{exp ww}
\tZ( \phi^{\infty}, \chi)(v)=\sum_{w\vert v} \langle {}^{\qqq}
\tZ_{*}(\phi^{\flat, \infty})1, {\rm T}(\sigma^{\vee})^{\rm t} t_{\chi}\rangle_{\ell,w}.\eeq
For $w\nmid p$, we may  move the correspondence  ${\rm T}(\sigma^{\vee})^{\rm t}$ back to the left entry by interpreting the resulting pairing similarly to \cite{yzz}. Namely,    the local height pairing of two degree-zero divisors $D_{1}$, $D_{2}$ on $X$ is, up to a factor $\ell(\vpi_{w})$, the intersection multiplicity of flat extensions of $D_{1}$, $D_{2}$ to an integral model  (see [I, Proposition 4.2.2]). In turn, this  arithmetic intersection pairing extends to divisors of arbitrary degree with disjoint supports by considering \emph{$\widehat{\xi}$-admissible} (rather than flat) extensions as in \cite[\S~7.1]{yzz}. As a result, the pairing 
$$\langle {}^{\qqq}\tZ_{*}(\phi^{ \infty})1,  t_{\chi}\rangle_{\ell,w}$$
is well-defined and it equals the $w$-term in \eqref{exp ww}. The fact that $t_{\chi}$ may not have degree zero introduces a term given by pairing with the Hodge class $\widehat{\xi}$, which however vanishes under our assumptions as in \cite[Proposition 7.3.3]{yzz}. Thus the expression of [I, (8.2.1)]    for $\tZ( \phi^{\infty}, \chi)(v)$ is still valid and  Theorem 3.6.1 continues to hold under our assumptions. 

Theorem \ref{MT} is therefore still reduced   to   Proposition \ref{kill lp}. For each nonsplit $v\vert p$,  fix $m=m_{v}\geq r$  which is a multiple of the order of $\vpi_{v}$ in the set \eqref{ss pts} below. 
Define an operator $\cR_{v}:= \Up_{v, *}^{m_{v}}-1$.  We will prove the following.

 \begin{prop}\label{is crit corr}  Let $v\vert p$. Under our running assumptions, the element
$$\cR_{v}\tZ(\phi^{\infty}, \chi)(v)\in \baar{\bf S}{}'$$
   is $v$-critical in the sense of \eqref{defcrit}.
 \end{prop}
Since $\lf\circ \cR_{v}= (\alpha_{v}^{m}-1)\lf$, and $\alpha_{v}^{m}-1\neq0$ by our assumptions, the proposition still implies Proposition \ref{kill lp}.

\subsection{Decay of intersection multiplicities}
We prove Proposition \ref{is crit corr}. We fix an inert place $v$ of $F$, and denote by $w$ its extension to $E$.

Given our assumption that $\chi_{v}$ is unramified, we consider the action of $\OO_{E,w}^{\ts}$ on CM points; for the set $\Xi(\vpi_{v}^{r})_{a}$ of  Lemma 4.1.3, we have
$$[\Xi(\vpi_{v}^{r})_{a}U_{F, v}^{\circ}]_{U}=  {\rm rec}_{E_{w}}(\OO_{E,w}^{\ts}/
\OO_{F,v}^{\ts}(1+\vpi_{v}^{r+s}\OO_{E,w})) [x(\baar{b}_{a})]_{U},$$
where the Galois action is faithful, and  $\baar{b}_{a}$ 
is any element of
\beq \lb{theset C}
q_{v}^{-1}( 1-a(1+\vpi_{v}^{r}\OO_{F,v})) / (1+\vpi_{v}^{r+s}\OO_{E,w})
\eeq
In fact, let $\sqrt{\ {} }$ be the principal square root defined in a neighbourhood of $1\in \OO_{F, v}$. Then, if $v(a)\geq1$ (or $v(a)\geq 2 $ if $v\vert 2$),  we may and do fix  $\baar{b}_{a}$ to be the class of
$${b}_{a}:=[\sqrt{1-a}].$$

Correspondingly, we define
$$H_{00}$$ to
 be the finite abelian extension of $E$ with norm group $U_{F}^{\circ} U_{T}^{v}\OO_{E,v}^{\ts}$. It is contained in the extension $H_{0}$  defined before Proposition 4.1.4, and it is unramified at $w$. The study of intersection multiplicities of \S~4.3 then needs to take place in $\X$, the base change to $H_{00, \baar{w}}$ of the integral model $\X^{\natural}/\OO_{F_{v}}$ of $X_{U}$  defined by Carayol  (we are renewing the notation: the model $\X$ considered in \S~4 is no longer in use). Note that under our assumption, $H_{00, \baar{w}/F_{v}}$ is unramified, so that $\X$ is still regular. 

Consider Proposition 4.3.3. As noted above, its statement need  to be corrected by replacing (4.3.2)  by 
\beq\lb{int for C} (\baar{z}\cd D) = 
c[\kappa(y):\kappa]  q_{F,v}^{s}+ \rho (V\cdot  D')_{\X_{y}},
\eeq
where $V=V(z)$, $\rho=\rho(z)$.  (The proof goes through verbatim in our renewed setup.)
 
 The following is the new ingredient needed.
\begin{prop} \lb{up-1}
The sequence 
$$(V_{s}, \rho_{s}):=\left((V,\rho)([x(\baar{b}_{a\vpi_{v}^{ms}})]_{U})\right)_{s\in \N}$$
is eventually constant.
\end{prop}
\begin{proof}
We first consider $V_{s}$. By construction, it is the irreducible component of $\X_{\kappa}$ maximising the intersection multiplicity  with the closure of the image $z_{s}\in X_{H_{00, \baar{w}}}$ of $[x(\baar{b}_{a\vpi_{v}^{ms}})]_{U}$ . Here $\kappa $ is the residue field of $H_{00, \baar{w}}$. However, the irreducible components of $\X_{\kappa}$ are already defined over the residue field $\kappa^{\natural}$ of $F_{v}$. Therefore $V_{s}$ is the base-change of the component 
${V}_{s}^{\natural}\subset \X_{\kappa^{\natural}}^{\natural}$
  maximising the intersection multiplicity with the closure of the image $z_{s}^{\natural}\in X_{F_{v}}$ of $z_{s}$.

We explicitly compute $V_{s}^{\natural}$ in terms of the (equivalent) notions of geometric and algebraic \emph{basins} of irreducible components introduced in \cite{equi}. In fact, $V_{s}^{\natural}$ is, essentially by definition, the component through $y$ to whose (geometric) basin the point $z_{s}^{\natural}$ belongs.
First, recall from \cite{carayol, equi} that: 
\begin{itemize}
\item the supersingular points in $\X_{\kappa}$ are parametrised by 
\beq\lb{ss pts}
B(v)^{\ts}\bks \B^{v\infty, \ts}\ts F_{v}^{\ts}/(U^{v}\ts q(U_{v}));\eeq
\item the irreducible components of $\X_{\kappa^{\natural}}$ are parametrised by $(\OO_{F_{v}}/\vpi_{v}^{r}\OO_{F_{v}})$-lines $L\subset (\vpi_{v}^{-r}\OO_{F_{v}}/\OO_{F_{v}})^{2}$;
\item to a CM-by-$E$ point  $z\in X_{F_{v}}$ with sufficiently large conductor is attached an $F_{v}$-isomorphism  $\tau\colon E_{w}\to F_{v}^{2}$, normalised so that 
\beq\lb{n tau}\OO_{F,v}^{2}\subset \tau(\OO_{E,w})\not\supset \vpi_{v}^{-1}\OO_{F, v}^{2},\eeq
and a corresponding line $L(\tau)= [\tau(\OO_{E_{w}})]\subset(\vpi_{v}^{-r}\OO_{F_{v}}/\OO_{F_{v}})^{2}$. 
\end{itemize}
Then in order to show the eventual constancy of $V_{s}^{\natural}$ we need to show that, for $y_{s}$ the reduction of $z_{s}$ and  $\tau_{s}$ the invariant attached to $z_{s}^{\natural}$, we have $y_{s+1}=y_{s}$ and  $L(\tau_{s+1})=L(\tau_{s})$ (for any sufficiently large $s$). 

We have $[x(\baar{b}_{a\vpi_{v}^{m(s+1)}})]_{U}= [x(\baar{b}_{a\vpi_{v}^{ms}}) h)]_{U}$ where, setting $b_{s}:= b_{a\vpi_{v}^{ms}}$,
$$h= (1+{\rm j}  b_{s})^{-1}(1+{\rm j} b_{s+1})= 
{1\over a\vpi_{v}^{ms}}  \twomat {(1-b_{s})(1+b_{s})}     {[b_{s+1}(1-b_{s}) - b_{s}(1-b_{s+1})] {\rm T}} 
{} {(1+b_{s}) (1-b_{s+1})} \quad \in B_{v}^{\ts}=\GL_{2}(F_{v}). 
 $$
The group $B_{v}^{\ts}$ acts on the supersingular points via the map to  the group \eqref{ss pts} induced by the reduced norm $q$. By construction, $q(h)=\vpi_{v}^{m}$ has trivial image there, thus $y_{s+1}=y_{s}$. 

We have $b_{s}=1-2^{-1}a\vpi_{v}^{ms} +O(\vpi_{v}^{ms})$, so that 
$$h \equiv \twomat 1 {{\rm T}/2} {} {0} \pmod{ \vpi_{v}^{m}}.$$ By the construction in \cite[(1.2.1)]{equi}, the group $B_{v}^{\ts}$ acts on the invariant $F_{v}^{\ts}\tau$ via left multiplication by $h^{\rm t}$. Recalling the normalisation \eqref{n tau}, 
we then have 
$$L(\tau_{s+1})= L\left( \twomat c {}{c{\rm T}/2} {0} \tau_{s}\right),$$ where $c\in F_{v}^{\ts}$ is such that the matrix is integral and not divisible by $\vpi_{v}$. But this line is just the one spanned by ${1 \choose 0}$ (note that $\tau_{s}$, as a surjective map to $F_{v}^{2}$, cannot be annihilated by a nonzero matrix over $F_{v}$). Thus $V_{s}$ is constant for $s\geq2$.

We now show the eventual constancy of $\rho_{s}$. In fact, for large enough $s$ we have $\rho_{s} = { (\baar{z}_{s}\cdot \Delta) / (V_{s} \cdot \Delta)}$ for any  divisor $\Delta$ whose support does not contain $V_{s}$.  We take $\Delta={\rm q}^{*}\Delta_{0}$, where $\X_{0}$ is as in the beginning of \S~\ref{43}, ${\rm q}\colon \X \to \X_{0, \OO_{H_{00}, \baar{w}}}$ is the projection, and $\Delta_{0}$ is the  Zariski closure of the canonical lift of $y:=y_{s}=y_{s+1}$. By the projection formula, and  with the notation of  the proof of Lemma \ref{int D}, 
 the intersection $(\baar{z}_{s}\cdot \Delta)$ is a constant multiple of 
\beq\lb{inttt}
(\baar{{\rm q}(z_{s})}\cdot \Delta_{0})_{\X_{0, \OO_{E, \baar{w}}^{\rm un}}, y}=\dim_{k} \OO_{E, \baar{w}}^{\rm un}\llb u\rrb / (\nu_{s}, u).\eeq
Here, by \cite{gross}, the local defining  equation of   the canonical lift $ \Delta_{0}$ is $u=0$, and for the quasicanonical lift $\baar{{\rm q}(z_{s})}$ it is  $\nu_{s}(u)=0$ for an Eisenstein polynomial $\nu_{s}$.
Thus \eqref{inttt} equals~$1$ independently of $s$. This completes the proof.
\end{proof}
The following replaces Corollary \ref{cococo}.

\begin{coro}\lb{cococo corr} If $D\in \Div^{0}(X_{H_{0}})_{L}$ is any degree-zero divisor,  then  for all sufficiently large $s$ and all $a$,
\beq 
\lb{eq up-1}
 m_{\baar{w}}( \wtil{Z}_{a\vpi^{m(s+1)}}(\phi^{\infty}) [1]_{U}, D) - m_{\baar{w}}( \wtil{Z}_{a\vpi^{ms}}(\phi^{\infty}) [1]_{U}, D) 
=O(q_{F,v}^{ms}) 
\eeq
in $L$, where the implied constant can be fixed independently of $a$ and $s$.
\end{coro}
\begin{proof} Let $\widehat{D}$ be a flat extension of $D$ to a divisor on $\X$ (with coefficients in $L$),
and abbreviate $Z_{a,s}:=  \wtil{Z}_{a\vpi^{ms}}(\phi^{\infty}) [1]_{U}$. Then  by the corrected Proposition 4.3.3,
\beq 
\lb{eq up-12}
m_{\baar{w}}( Z_{a,s}, D)= (\baar{Z}_{a,s}\cd \widehat{D}) = A_{s} q_{F,v}^{ms}+ \sum_{i} \lm_{i,s}(V_{i,s} \cdot D'_{s})
\eeq
for some vertical components $V_{i,s}\subset \X $ and some  $A_{s}$, $\lm_{i,s}\in L$; here we have written $\widehat{D}=c_{s}\X_{\kappa}+ D'_{s}$ where $D'_{s} $ is a divisor whose support does not contain $V_{s}$.   By   Proposition \ref{up-1} (transported  by  Hecke correspondences away from $v$), all terms indexed by $s$ are in fact eventually independent of $s$; thus the second term of \eqref{eq up-12} gives vanishing contribution to \eqref{eq up-1}. As remarked in Corollary \ref{cococo}, the constant $A=A_{s}$ is independent of $a$ as well.
\end{proof}

Then the argument of the proof of Proposition \ref{is crit} at the very end of the paper  goes through to prove Proposition \ref{is crit corr}, with the following modifications: we apply the operator $(\cR_{v}^{\rm seq}\star)_{s}:= \star_{m(s+1)}- \star_{ms}$ to  \eqref{longg} and \eqref{pollo} (each viewed as a sequence $\star$ in $s$), and use Corollary \ref{cococo corr} instead of Corollary \ref{cococo}.
 \medskip

 \textbf{A further erratum to \cite{1}}

In Lemma 8.2.1 and in the Proof of Proposition 8.2.2, one should read `$F^\times \mathbf{A}^{S_1 \infty, \times}$' in place of $\mathbf{A}^{S_1 \infty, \times}$. (I am grateful to Yangyu Fan for pointing this out.)

 \medskip
 \emph{{\bf References} for Appendix C}
 
{\small [Ram] \quad {Dinakar Ramakrishnan}, \emph{A theorem on $GL(n)$ \`a la Tchebotarev}, preprint.}

\bigskip

\bigskip
\bigskip
\bigskip
\bigskip
\bigskip

\backmatter
\addtocontents{toc}{\medskip}

\end{document}